\documentclass[twoside,10pt]{amsart}
\usepackage{a4wide}
\usepackage[utf8]{inputenc}
\usepackage[english]{babel}
\usepackage{csquotes}
\usepackage{color,graphicx}
\usepackage{amssymb,amsfonts,amsmath,amsthm,amscd,epsfig,hyperref}
\usepackage[all]{xy}
\usepackage{tensor}
\usepackage{microtype}
\usepackage{tikz-cd}
\usepackage{float}
\usepackage{graphicx}
\usepackage{amssymb}
\usepackage{amsmath}
\usepackage{amsthm}
\usepackage{indentfirst}
\usepackage{tikz}
\usetikzlibrary{cd}
\usepackage[numbers,square]{natbib}

\newcounter{contador}
\numberwithin{contador}{section}

\newtheorem{theorem}[contador]{Theorem}
\newtheorem{prop}[contador]{Proposition}
\newtheorem{lemma}[contador]{Lemma}
\newtheorem{corollary}[contador]{Corollary}

\theoremstyle{definition}
\newtheorem{defi}[contador]{Definition}
\newtheorem{obs}[contador]{Remark}
\newtheorem{exe}[contador]{Example}

\newcommand{\G}{\mathcal{G}}

\title{Groupoid Twisted Partial Actions}

\author[Bemm, Lautenschlaeger, Tamusiunas]
{Laerte Bemm, Wesley G. Lautenschlaeger, Tha\'{\i}sa Tamusiunas}

\address{ Departamento de Matem\'{a}tica, Universidade Estadual do Maring\'{a}, Av. Colombo, 5790 - Campus Universit\'{a}rio, 87020-900, Maring\'{a}-PR, Brazil} \email{lbemm2@uem.br}
\address{Instituto de Matem\'{a}tica, Universidade Federal do Rio Grande do Sul,  Av. Bento Gon\c{c}alves, 9500, 91509-900. Porto Alegre-RS, Brazil} \email{wesleyglautenschlaeger@gmail.com, thaisa.tamusiunas@gmail.com}

\begin{document}

\begin{abstract}
   We introduce twisted partial actions of groupoids and establish necessary and sufficient conditions for the existence of a globalization. We prove that the crossed products associated with a twisted partial action and with its globalization are Morita equivalent. Finally, we extend partial projective representations and partial Schur multipliers to the groupoid setting and describe their interaction with twisted partial actions of groupoids.
\end{abstract}
  
\maketitle

\vspace{0.5 cm}



\section{Introduction}
Twisted partial actions of groups on abstract rings and the corresponding crossed products were introduced by Dokuchaev, Exel, and Simón in \cite{dokuchaev2008crossed}. In \cite{dokues}, the same authors established criteria for the existence of a globalization for a given twisted partial action of a group on a ring. With the aim of intrinsically relating partial representations with partial actions, in \cite{dokuchaev2010partialprojectiverep} and \cite{dokuchaev2012partialprojectiverep2} the authors introduced the concept of a partial projective representation of a group. In \cite{dokuchaev2010partialprojectiverep}, it was shown that the factor sets of partial projective representations over a field $\mathbb{K}$ are exactly the $\mathbb{K}$-valued twistings of crossed products by partial actions.

A groupoid is a small category in which every morphism has an inverse. When a groupoid has only one object, it reduces to a group. We recall in Subsection 2.1 a purely algebraic definition of groupoids that is convenient for our purposes. Partial actions of groupoids were originally motivated by the generalization of the notion of partial Galois extensions of commutative rings and were first introduced by Bagio and Paques in \cite{bagio}. Since both groupoid partial actions and twisted partial actions of groups arise as particular cases of twisted partial actions of groupoids, it is natural to develop a unified algebraic theory for groupoid twisted partial actions.

The main purpose of this paper is to extend the notion of twisted partial actions from groups to groupoids and to analyze when a given twisted partial action of a groupoid admits a globalization, that is, when it can be realized as the restriction of a global action. While the resulting criteria are closely related to those in the group case, the proofs require more delicate arguments, since one must deal with ideals of ideals rather than ideals of the ambient ring, and since compositions of actions are only partially defined for certain composable pairs in the groupoid. We further show that, in the groupoid setting, there remains a meaningful interaction between partial projective representations and twisted partial actions, which is developed in Section 7.

The paper is organized as follows. We begin by recalling some basic notions and properties of groupoids and semigroupoids. In Section 3 we introduce twisted partial actions of groupoids on rings and their globalizations. In Theorem \ref{teoglobal} we establish necessary and sufficient conditions for a twisted partial action to be globalizable, generalizing \cite[Theorem 4.1]{dokues}. In Section 4 we define the twisted crossed product, and Theorem \ref{morita} shows that the crossed product of a (globalizable) twisted partial action is Morita equivalent to the crossed product of its globalization. In Section 5 we introduce groupoid partial projective representations and study the structure of the associated partial Schur multiplier. In Section 6 we define twisted partial actions of groupoids on semigroups, and in Proposition \ref{propcompalgmon} we relate twisted partial actions on rings to twisted partial actions on semigroups. Finally, in Section 7 we establish the connection between twisted partial actions and partial projective representations; see Theorem \ref{teoprincipal}.

By convention, throughout this paper rings and algebras are associative and not necessarily unital.

\section{Preliminaries}

The background on groupoids and semigroupoids is reviewed here.

\subsection{Groupoids}
Recall that a \emph{groupoid} is a nonempty set $\G$, equipped with a partially defined binary operation, which we will denote by concatenation, that satisfies the associative law and the condition that each element $g \in \G$ has a right and a left identity and an inverse $g^{-1}$. Given $g \in \G$, the right identity will be called \emph{domain} and the left identity will be called \emph{range} of $g$, and they will be denoted by $d(g)$ and $r(g)$, respectively. Hence, $d(g) = g^{-1}g$, $r(g) = gg^{-1}$ and $d(g) = r(g^{-1})$, for all $g \in \G$. This definition appears, for instance, in \cite[p. 78]{lawson1998inverse}. 

For all $g, h \in \G$, we write $\exists gh$ whenever the product $gh$ is defined. As it was observed in \cite[Lemma 3.1]{pata}, it is immediately from the definition that for all $g, h \in \G$, $\exists gh$ if and only if $d(g) = r(h)$, and in this case $d(gh) = d(h)$ and $r(gh) = r(g)$. It will be denoted by $\G^{(2)}$ the subset of the pairs $(g, h) \in \G \times \G$ such that $d(g) = r(h)$. 

An element $e \in \G$ is called an identity of $\G$ if $e = d(g) = r(g^{-1})$, for some $g \in \G$. It will be denoted by $\G_0$ the set of all identities of $\G$. 

\subsection{Semigroupoids}
We briefly recall the notions of semigroupoids, inverse semigroupoids, and inverse categories, which provide the natural framework for partial projective representations and will be used in Section 5 to relate such representations to twisted partial actions.

% A projective representation (and a partial projective representation) is a map from a semigroupoid to a monoid satisfying certain properties; see details in Section 5. When the semigroupoid is, in particular, a groupoid, there is a connection between partial projective representations and twisted partial actions. Thus, in this subsection we introduce the notions of semigroupoid, inverse semigroupoid, and inverse category, and we will discuss these structures again in Section 5.

A \emph{semigroupoid} $\mathcal{S}$ is a set endowed with a partially defined binary operation that is associative whenever it is defined. We denote by $\mathcal{S}^{(2)} \subseteq \mathcal{S} \times \mathcal{S}$ the set of composable pairs in $\mathcal{S}$. We often write $\exists\, st$ to indicate that $(s,t) \in \mathcal{S}^{(2)}$. A \emph{semigroupoid homomorphism} $\varphi \colon \mathcal{S} \to \mathcal{T}$ is a map such that, whenever $(s,t) \in \mathcal{S}^{(2)}$, one has $(\varphi(s),\varphi(t)) \in \mathcal{T}^{(2)}$ and $\varphi(st) = \varphi(s)\varphi(t)$.

Let $\mathcal{S}$ be a semigroupoid. A subset $\mathcal{I} \subseteq \mathcal{S}$ is called a \emph{left ideal} if, whenever $t \in \mathcal{I}$ and $s \in \mathcal{S}$ are such that $(s,t) \in \mathcal{S}^{(2)}$, it follows that $st \in \mathcal{I}$. Right ideals and two-sided ideals are defined analogously, and a two-sided ideal will simply be called an \emph{ideal}.

A \emph{category} is a semigroupoid $\mathcal{C}$ such that for each $x \in \mathcal{C}$ there exist unique identities $\mathrm{r}(x), \mathrm{d}(x) \in \mathcal{C}$ satisfying $(\mathrm{r}(x),x), (x,\mathrm{d}(x)) \in \mathcal{C}^{(2)}$. With this notation, given $x,y \in \mathcal{C}$, the pair $(x,y)$ is composable if and only if $\mathrm{d}(x) = \mathrm{r}(y)$. We denote by $\mathcal{C}_0$ the set of identities of $\mathcal{C}$ and by $\mathcal{C}^{(n)}$ the set of $n$-tuples $(x_1,\ldots,x_n)$ for which the product $x_1 x_2 \cdots x_n$ is defined.

An \emph{inverse semigroupoid} is a semigroupoid $\mathcal{S}$ such that for every $s \in \mathcal{S}$ there exists a unique element $t \in \mathcal{S}$ with $(s,t), (t,s) \in \mathcal{S}^{(2)}$ and $sts = s$, $tst = t$. We denote this inverse by $t = s^{-1}$. An \emph{inverse category} is an inverse semigroupoid that is also a category. In particular, every groupoid is an inverse category.

\section{Twisted Partial Actions on Rings and Globalization}

Let $\mathbb{K}$ be a commutative associative unital ring, which will be the base ring for our algebras. Recall from \cite[p. 4138]{dokues} that the multiplier algebra $\mathcal{M}(A)$ of an associative, not necessarily unital, $\mathbb{K}$-algebra $A$ is the set: $$\mathcal{M}(A) = \{(R,L) \in End(_AA) \times End(A_A); (aR)b = a(Lb), \forall a, b \in A\}, $$with component-wise addition and multiplication.

If $I$ is a unital ideal of $A$, then $I \simeq \mathcal{M}(I)$ via $a \to (R_a,L_a)$, where $a \in I$ \cite[Proposition 2.3]{dokuexel}. Recall that the multiplier $(R_a,L_a)$ is given by $(b)R_a = ba$, $L_a(b) = ab$, for all $a,b \in I$. 

        \begin{defi}\label{def21} A \emph{groupoid twisted partial action of a groupoid} $\G$ \emph{on a ring} $R$ is a triple $$(\{D_g\}_{g \in \mathcal{G}}, \{\alpha_g\}_{g \in \mathcal{G}}, \{w_{g,h}\}_{(g,h) \in \mathcal{G}^{(2)}})$$ where for each $g \in \G$, $D_{r(g)}$ is an ideal of $R$, $D_g$ is an ideal of $D_{r(g)}$, $\alpha_g: D_{g^{-1}} \rightarrow D_g$ is an isomorphism of rings and for each $(g,h) \in \G^{(2)}$, $w_{g,h}$ is an invertible element from $\mathcal{M}(D_gD_{gh})$, satisfying, for all $(g,h,k)$ in $\G^{(3)}$: \begin{itemize}
                    \item [\emph{(i)}] $D_g^2 = D_g$, $D_gD_h = D_hD_g$;
                    \item [\emph{(ii)}]  $R = \sum_{e \in \G_0} D_e$ and $\alpha_e$ is the identity map of $D_e$, for all $e \in \G_0$;
                    \item [\emph{(iii)}] $\alpha_g(D_{g^{-1}}D_h) = D_gD_{gh}$;
                    \item [\emph{(iv)}] $\alpha_g \circ \alpha_h (a) = w_{g,h}\alpha_{gh}(a)w_{g,h}^{-1}$, for all $a \in D_{h^{-1}}D_{h^{-1}g^{-1}}$;
                    \item [\emph{(v)}] $w_{r(g),g} = w_{g,d(g)} = id_{\mathcal{M}(D_g)}$;
                    \item [\emph{(vi)}] $\alpha_g(aw_{h,k})w_{g,hk} = \alpha_g(a)w_{g,h}w_{gh,k}$, for all $a \in D_{g^{-1}}D_hD_{hk}$.
                  \end{itemize}
\end{defi}

Observe that if each $D_g$ is unital with identity $1_g$, then $\mathcal{M}(D_g) \cong D_g$ and in this case we identify $id_{\mathcal{M}(D_g)}$ with $1_g$.

\begin{exe}\label{ex1}
    Let $R = \bigoplus\limits_{i=1}^4 \mathbb{K}e_i$, where the $e_i$'s are central orthogonal idempotents such that $1_R = e_1 + e_2 + e_3 + e_4$. Consider $\G = \{ g, g^{-1}, r(g), d(g) \}$.
    
    An example of groupoid twisted partial action of $\G$ on $R$ is given by taking
    \begin{align*}
        D_{r(g)} = \mathbb{K}e_1 \oplus \mathbb{K}e_2, \; D_{d(g)} = \mathbb{K}e_3 \oplus \mathbb{K}e_4, \;      D_{g} = \mathbb{K}e_1, \; D_{g^{-1}} = \mathbb{K}e_3, \\    
        \alpha_{r(g)} = \text{Id}_{D_{r(g)}}, \; \alpha_{d(g)} = \text{Id}_{D_{d(g)}}, \;     \alpha_g(xe_3) = xe_1, \; \alpha_{g^{-1}}(xe_1) = xe_3.
    \end{align*}
    
    This gives us a partial groupoid action of $\G$ on $R$ in the sense of \cite{bagio}. If we consider the twisting
    \begin{align*}
        w_{r(g),r(g)}  = 1_{r(g)} = e_1 + e_2, \; &w_{d(g),d(g)}  = 1_{d(g)} = e_3 + e_4 \\
        w_{r(g),g}  = w_{g,d(g)} = 1_g = e_1, \; &w_{d(g),g^{-1}}  = w_{g^{-1},r(g)} = 1_{g^{-1}} = e_3 \\
        w_{g,g^{-1}} = -e_1, \; & w_{g^{-1},g} = -e_3,
    \end{align*}
    we have a groupoid twisted partial action $\alpha = (\{D_g\}_{g \in \mathcal{G}}, \{\alpha_g\}_{g \in \mathcal{G}}, \{w_{g,h}\}_{(g,h) \in \mathcal{G}^{(2)}})$ of $\G$ on $R$. In fact, the conditions (iv) and (v) are clear, so it only remains to us to prove (vi). The elements in $\G^{(3)}$ that involve $d(g),r(g)$ are direct. Consider the triples $(g,g^{-1},g)$ and $(g^{-1},g,g^{-1})$. Then
    \begin{align*}
        \alpha_g(xe_3w_{g^{-1},g})w_{g,g^{-1}g} = \alpha_g(xe_3w_{g^{-1},g})w_{g,d(g)} = \alpha_g(-xe_3)e_1 = -xe_1
    \end{align*}
    and
    \begin{align*}
        \alpha_g(xe_3)w_{g,g^{-1}}w_{gg^{-1},g} = \alpha_g(xe_3)w_{g,g^{-1}}w_{r(g),g} = \alpha_g(xe_3)(-e_1)e_1 = -xe_1,
    \end{align*}
    and the other triple is handled in similar way.
 \end{exe} 
 
 A twisted partial action $\beta = (E_g, \beta_g, u_{g,h})_{(g,h) \in \G^{(2)}}$ is said \emph{global} if and only if $E_g = E_{r(g)}$ for all $g \in \G$. 
 
\begin{exe}
Consider $R$ and $\G$ as in Example \ref{ex1}. Define:
    \begin{align*}
        E_g = E_{r(g)} = \mathbb{K}e_3 \oplus \mathbb{K}e_4, &\; E_{g^{-1}} = E_{d(g)}  = \mathbb{K}e_1 \oplus \mathbb{K}e_2, \\
        \beta_{r(g)} = \text{Id}_{E_g}, &\;
        \beta_{d(g)} = \text{Id}_{E_{g^{-1}}}, \\
        \beta_g(xe_1 + ye_2) = xe_3 + ye_4, &\;
        \beta_{g^{-1}}(xe_3 + ye_4) = xe_1 + ye_2.
    \end{align*}
    
    Fix $a,b \in \mathcal{U}(\mathbb{K})$ and let
    \begin{align*}
        u_{r(g), r(g)} = u_{g,d(g)} = u_{r(g),g} = 1_g = e_3 + e_4, &\;  u_{d(g),d(g)} = u_{g^{-1},r(g)} = u_{d(g),g^{-1}} = 1_{g^{-1}} = e_1 + e_2, \\
        u_{g,g^{-1}} = ae_3 + be_4 \in \mathcal{U}(E_{r(g)}), &\; u_{g^{-1},g} = ae_1 + be_2 \in \mathcal{U}(E_{d(g)}).
    \end{align*}
    
    Then $\beta = (E_g, \beta_g, u_{g,h})_{(g,h) \in \G^{(2)}}$ is a groupoid twisted global action of $\G$ on $R$.
\end{exe}

We can obtain examples of twisted partial actions by restricting a global one, in a standard way. Given a groupoid twisted global action $\beta = (\{E_g\}_{g \in \mathcal{G}}, \{\beta_g\}_{g \in \mathcal{G}}, \{w_{g,h}\}_{(g,h) \in \mathcal{G}^{(2)}})$ of $\G$ on a (non-necessarily unital) ring $S$ such that each $E_g$ is unital with identity $\overline{1}_g$, one can restricts $\beta$ to a two-sided ideal $R$ of $S$ such that $R$ has identity $1_R$. Let $D_{r(g)} = R \cap E_{r(g)} = R \cdot E_{r(g)}$, which is an ideal of $E_{r(g)}$, $g \in \G$. Note that $D_{r(g)}$ has identity $1_{r(g)} = 1_R\overline{1}_g$.

Let $D_g = D_{r(g)} \cap \beta_g(D_{d(g)}) = D_{r(g)} \cdot \beta_g(D_{d(g)})$, which is an ideal of $D_{r(g)}$, for all $g \in \G$. Observe that $1_g = 1_{r(g)}\beta_g(1_{d(g)})$. Putting $\alpha_g = \beta_g\mid_{D_{g^{-1}}}$, the itens (i), (ii) and (iii) of Definition \ref{def21} are clearly satisfied. Once each $D_g$ is unital, we got $\mathcal{M}(D_g) \cong D_g$, for all $g \in \G$. So we can consider $w_{g,h} \in D_gD_{gh}$. Defining $w_{g,h} = u_{g,h}1_g1_{gh}$, we have that (iv), (v) and (vi) are also satisfied, then we obtain, in fact, a groupoid twisted partial action on $R$.

We have given a global action and constructed a partial one. We are now interested in determining under what circumstances a given partial action can be obtained as a restriction of a global one.

\begin{defi} \label{def22} A groupoid twisted global action $\beta = (\{E_g\}_{g \in \mathcal{G}}, \{\beta_g\}_{g \in \mathcal{G}}, \{w_{g,h}\}_{(g,h) \in \mathcal{G}^{(2)}})$ of a groupoid $\G$ on a non-necessarily unital ring $S$ is said to be a \emph{globalization} for the partial action $\alpha = (\{D_g\}_{g \in \mathcal{G}}, \{\alpha_g\}_{g \in \mathcal{G}}, \{w_{g,h}\}_{(g,h) \in \mathcal{G}^{(2)}})$ of $\G$ on $R$ if, for each $e \in \G_0$, there exists a monomorphism $\varphi_e: D_e \rightarrow E_e$ such that: \begin{itemize}
                    \item [\emph{(i)}] $\varphi_e(D_e)$ is an ideal of $E_e$;
                    \item [\emph{(ii)}] $E_g = \sum_{r(h) = r(g)} \beta_h (\varphi_{d(h)}(D_{d(h)}))$;
                    \item [\emph{(iii)}] $\varphi_{r(g)}(D_g) = \varphi_{r(g)}(D_{r(g)}) \cap \beta_g(\varphi_{d(g)}(D_{d(g)}))$;
                    \item [\emph{(iv)}] $\beta_g \circ \varphi_{d(g)}(a) = \varphi_{r(g)} \circ \alpha_g(a)$, $\forall a \in D_{g^{-1}}$;
                    \item [\emph{(v)}] $\varphi_{r(g)}(aw_{g,h}) = \varphi_{r(g)}(a)u_{g,h}$, $\, \,$ $\varphi_{r(g)}(w_{g,h}a) = u_{g,h}\varphi_{r(g)}(a)$, $\forall a \in D_gD_{gh}$.
                  \end{itemize} In this case, we say that $\alpha$ is \emph{globalizable}.

\end{defi}

Suppose that each $D_{e}$, $e \in \G_0$, is unital. If $\alpha$ is a globalizable groupoid twisted partial action, then each $D_g$, $g \in \G$, is unital, since $\varphi_{r(g)}(D_g) = \varphi_{r(g)}(D_{r(g)}) \cap \beta_g(\varphi_{d(g)}(D_{d(g)}))$. In the non-twisted groupoid case \cite[Theorem 2.1]{bagio}, the converse is true. However, in the twisted case, even when the groupoid is a group \cite[Theorem 4.1]{dokues}, an extendability property of the partial twisting is required to ensure the globalization.

\begin{theorem}\label{teoglobal}
Let $\alpha = (\{D_g\}_{g \in \mathcal{G}}, \{\alpha_g\}_{g \in \mathcal{G}}, \{w_{g,h}\}_{(g,h) \in \mathcal{G}^{(2)}})$ be a groupoid twisted partial action of $\G$ over $R$ such that each $D_g$, $g \in \G$, is a unital ring. Then $\alpha$ admits globalization $\beta = (\{E_g\}_{g \in \mathcal{G}}, \{\beta_g\}_{g \in \mathcal{G}}, \{w_{g,h}\}_{(g,h) \in \mathcal{G}^{(2)}})$ if and only if for each pair $(g,h) \in \G^{(2)}$, there exists an invertible element $\tilde{w}_{g,h} \in \mathcal{U}(D_{r(g)})$ such that $\tilde{w}_{g,h}1_g1_{gh} = w_{g,h}$ and, for $k \in \G$ such that $(h,k) \in \G^{(2)}$, $$\alpha_g(\tilde{w}_{g,k}1_{g^{-1}})\tilde{w}_{g,hk} = 1_g\tilde{w}_{g,h}\tilde{w}_{gh,k}. \quad (\ast)$$
\end{theorem}
\begin{proof}
If $\alpha$ admits globalization $\beta$, take $\widetilde{w}_{g,h} = u_{g,h} \cdot 1_g,$ for $(g,h) \in \G^{(2)}$.

For the converse, let $\mathcal{F} = \mathcal{F}(\G,R) = \{f: \G \rightarrow R : f \text{ is a map} \}$. We also use the notation $f(g) = f|_g$, for $f \in \mathcal{F}, g \in \G$. For each $g \in \G$, we set 
\begin{itemize} 
\item $X_g = \{h \in \G : r(h) = r(g)\}$; 

\item $F_g = \{f \in \mathcal{F} : f(h) = 0, \text{ for all } \, h \notin X_g\}$; and 

\item $Y_g = \{f \in F_g : f(h) \in D_{d(h)}\}$.
\end{itemize} 

Clearly, $Y_g$ is an ideal of $\mathcal{F}$ and $Y_g = Y_{r(g)}$, since $F_g = F_{r(g)}$.

For $g \in \G$ and $f \in Y_{g^{-1}}$, let $\beta_g: Y_{g^{-1}} \rightarrow Y_g$ be the map given by
\begin{align*}
    \beta_g(f)|_h = \begin{cases}
        \tilde{w}_{h^{-1},g}f(g^{-1}h)\tilde{w}_{h^{-1},g}^{-1}, \text{ if } h \in X_g, \\
        0, \text{ otherwise.} 
        \end{cases}   
\end{align*}

Notice that if $h \in X_g$, $r(h) = r(g) = d(g^{-1})$, which implies that the product $g^{-1}h$ exists and $r(g^{-1}h) = r(g^{-1})$. Thus $\beta_g$ is well defined.

\noindent \textbf{Assertion:} $\beta_g$ is a ring isomorphism.

Indeed, $\beta_g$ is clearly a ring homomorphism. Now for $g \in \G$ and $f \in Y_{g^{-1}}$, define:
\begin{align*}
\beta_g^{-1}(f)|_h = \begin{cases}
\tilde{w}_{h^{-1}g^{-1},g}^{-1}f(gh)\tilde{w}_{h^{-1}g^{-1},g}, \text{ if } h \in X_{g^{{-1}}}, \\
 0, \text{ otherwise.} 
\end{cases}    
\end{align*}

\noindent (i) For $h \in X_{g^{-1}}$,
\begin{align*}
    \beta_g^{-1} \circ \beta_g (f)|_h & = \beta_g^{-1}(\beta_g (f))|_h = \tilde{w}_{h^{-1}g^{-1},g}^{-1}\beta_g(f)(gh)\tilde{w}_{h^{-1}g^{-1},g}\\ 
    & = \tilde{w}_{h^{-1}g^{-1},g}^{-1}\tilde{w}_{h^{-1}g^{-1},g}f(g^{-1}(gh))\tilde{w}_{h^{-1}g^{-1},g}^{-1}\tilde{w}_{h^{-1}g^{-1},g} \\ 
    & =  f((g^{-1}g)h) = f(d(g)h) = f(r(h)h) = f(h).
\end{align*}

\noindent (ii) For $h \in X_{g}$, 
\begin{align*}
    \beta_g \circ \beta_g^{-1} (f)|_h & = \beta_g(\beta_g^{-1} (f))|_h = \tilde{w}_{h^{-1},g}\beta_g^{-1}(f)(g^{-1}h)\tilde{w}_{h^{-1},g}^{-1}\\ 
    &= \tilde{w}_{h^{-1},g}\tilde{w}_{h^{-1}gg^{-1},g}^{-1}f(g(g^{-1}h))\tilde{w}_{h^{-1}gg^{-1},g}\tilde{w}_{h^{-1},g}^{-1} \\
    &=  f((gg^{-1})h) = f(r(g)h) = f(r(h)h) = f(h).
\end{align*}

Then $\beta_g$ is a ring isomorphism.

For $(g,h) \in \G^{(2)}$, define $u_{g,h} \in \mathcal{U}(Y_g)$ by
\begin{align*}
    u_{g,h}|_k = \begin{cases}
    \tilde{w}_{k^{-1},g}\tilde{w}_{k^{-1}g,h}\tilde{w}_{k^{-1},gh}^{-1}, \text{ if } k \in X_g, \\
        0, \text{ otherwise.} 
    \end{cases}   
\end{align*}
 
Observe that 
\begin{align*}
    u_{g,h}^{-1}|_k = \begin{cases}
        \tilde{w}_{k^{-1},gh}\tilde{w}_{k^{-1}g,h}^{-1}\tilde{w}_{k^{-1},g}^{-1}, \text{ if } k \in X_g \\
        0, \text{ otherwise.} 
    \end{cases}   
\end{align*}

We shall prove that $\beta = (Y_g, \beta_g, u_{g,h})_{(g,h)\in \G^{(2)}}$ is a groupoid twisted global action of $\G$ over $\mathcal{F}$. Since $Y_g = Y_{r(g)}$, it is enough to show that $\beta$ satisfies all the conditions of Definition \ref{def21}.

Let's start by showing the condition (vi) of Definition \ref{def21}. It is sufficient to verify that $\beta_g(u_{h,k})u_{g,hk} = u_{g,h}u_{gh,k}$, for $(g,h,k) \in \G^{(3)}$. For $x \in \G$, we have that $\beta_g(u_{h,k})u_{g,hk}|_x = \beta_g(u_{h,k})|_xu_{g,hk}|_x$ and $u_{g,h}u_{gh,k}|_x = u_{g,h}|_xu_{gh,k}|_x$. Then
\begin{align*}
 \beta_g(u_{h,k})|_xu_{g,hk}|_x
    & = \begin{cases}
        [\tilde{w}_{x^{-1},g}u_{h,k}(g^{-1}x)\tilde{w}_{x^{-1},g}^{-1}][\tilde{w}_{x^{-1},g}\tilde{w}_{x^{-1}g,hk}\tilde{w}_{x^{-1},ghk}^{-1}], \text{ if } x \in X_g, \\
        0, \text{ otherwise}
    \end{cases} \\
    & = \begin{cases}
        [\tilde{w}_{x^{-1},g}(\tilde{w}_{x^{-1}g,h}\tilde{w}_{x^{-1}gh,k}\tilde{w}_{x^{-1}g,hk}^{-1})\tilde{w}_{x^{-1},g}^{-1}] \\ 
        [\tilde{w}_{x^{-1},g}\tilde{w}_{x^{-1}g,hk}\tilde{w}_{x^{-1},ghk}^{-1}], \text{ if } x \in X_g, \\ 
        0, \text{ otherwise}
    \end{cases} \\
    & = \begin{cases}
        \tilde{w}_{x^{-1},g}\tilde{w}_{x^{-1}g,h}\tilde{w}_{x^{-1}gh,k}\tilde{w}_{x^{-1},ghk}^{-1}, \text{ if } x \in X_g, \\ 
        0, \text{ otherwise}
    \end{cases}
\end{align*}
and
\begin{align*}
    u_{g,h}|_xu_{gh,k}|_x     & = \begin{cases}
        [\tilde{w}_{x^{-1},g}\tilde{w}_{x^{-1}g,h}\tilde{w}_{x^{-1},gh}^{-1}][\tilde{w}_{x^{-1},gh}\tilde{w}_{x^{-1}gh,k}\tilde{w}_{x^{-1},ghk}^{-1}], \text{ if } x \in X_{gh} = X_g, \\ 
        0, \quad \text{ otherwise}
    \end{cases} \\
    & = \begin{cases}
        \tilde{w}_{x^{-1},g}\tilde{w}_{x^{-1}g,h}\tilde{w}_{x^{-1}gh,k}\tilde{w}_{x^{-1},ghk}^{-1}, \text{ if } x \in X_g, \\ 
        0, \text{ otherwise,}
    \end{cases}
\end{align*}
which completes the condition (vi).

Next we compute (iv) of Definition \ref{def21}. For any $f \in F_{h^{-1}}$, $(g,h) \in \G^{(2)}$ and $x \in \G$, we have
\begin{align*}
    (\beta_g \circ \beta_h(f))|_x & = \begin{cases}
        \tilde{w}_{x^{-1},g}\beta_h(f)(g^{-1}x)\tilde{w}_{x^{-1},g}^{-1}, \text{ if } x \in X_g, \\ 
        0, \text{ otherwise}
    \end{cases} \\
    & = \begin{cases}
        \tilde{w}_{x^{-1},g}\tilde{w}_{x^{-1}g,h}f(h^{-1}g^{-1}x)\tilde{w}_{x^{-1}g,h}^{-1}\tilde{w}_{x^{-1},g}^{-1}, \text{ if } x \in X_g, \\ 
        0, \text{ otherwise.}
    \end{cases}
\end{align*}

Furthermore,
\begin{align*}
    (u_{g,h}\beta_{gh}(f)u_{g,h}^{-1})|_x & = \begin{cases}
        [\tilde{w}_{x^{-1},g}\tilde{w}_{x^{-1}g,h}\tilde{w}_{x^{-1},gh}^{-1}][\tilde{w}_{x^{-1},gh}f(h^{-1}g^{-1}x)\tilde{w}_{x^{-1},gh}^{-1}] \\ 
        [\tilde{w}_{x^{-1},gh}\tilde{w}_{x^{-1}g,h}^{-1}\tilde{w}_{x^{-1},g}^{-1}], \text{ if } x \in X_g, \\ 
        0, \text{ otherwise}
    \end{cases} \\
    & = \begin{cases}
        \tilde{w}_{x^{-1},g}\tilde{w}_{x^{-1}g,h}f(h^{-1}g^{-1}x)\tilde{w}_{x^{-1}g,h}^{-1}\tilde{w}_{x^{-1},g}^{-1}, \text{ if } x \in X_g, \\ 
        0, \text{ otherwise}.
    \end{cases}
\end{align*}

Consequently, $\beta_g \circ \beta_h(f) = u_{g,h}\beta_{gh}(f)u_{g,h}^{-1}$ and the equality (iv) holds.

Since $w_{r(g),g} = w_{g,d(g)} = 1_g$ and, by the hypothesis, $\tilde{w}_{g,h}1_g1_{gh} = w_{g,h}$, it follows that $\tilde{w}_{r(g),g} = \tilde{w}_{g,d(g)} = 1_{r(g)}$.  

To see the condition (ii) of the Definition \ref{def21}, for $g \in \G$ and $f \in Y_{g^{-1}}$, take 
\begin{align*}
    \beta_{d(g)}(f)|_h & = \begin{cases}
        \tilde{w}_{h^{-1},d(g)}f(d(g)h)\tilde{w}_{h^{-1},d(g)}^{-1}, \text{ if } h \in X_{d(g)},\\ 
        0, \text{ otherwise}
    \end{cases} \\
    & = \begin{cases}
        \tilde{w}_{h^{-1},r(h)}f(r(h)h)\tilde{w}_{h^{-1},r(h)}^{-1}, \text{ if } h \in X_{d(g)}, \\ 
        0, \text{ otherwise}
    \end{cases} \\
    & = \begin{cases}
        \tilde{w}_{h^{-1},d(h^{-1})}f(h)\tilde{w}_{h^{-1},d(h^{-1})}^{-1}, \text{ if } h \in X_{d(g)}, \\ 
        0, \quad \text{otherwise}
    \end{cases} \\
    & = \begin{cases}
        1_{r(h^{-1})}f(h)1_{r(h^{-1})}, \text{ if } h \in X_{d(g)}, \\ 
        0, \text{ otherwise}
    \end{cases} \\
    & = \begin{cases}
        1_{d(h)}f(h)1_{d(h)}, \text{ if } h \in X_{d(g)}, \\ 
        0, \text{ otherwise}
    \end{cases} \\
    & = \begin{cases}
        f(h), \text{ if } h \in X_{d(g)}, \\
        0, \text{ otherwise}
    \end{cases} \\
    & = f(h).
\end{align*}

Once $f \in Y_{g^{-1}}$, it implies that $f \in F_{g^{-1}}$, that is, $f(h) = 0$ for all $h \notin X_{d(g)}$, and that is why the last equality holds. Then $\beta_{d(g)}$ is the identity map of $Y_{g^{-1}}$ and the condition (ii) is satisfied.

The conditions (i) and (iii) of Definition \ref{def21} are straightforward. It remains to ver\text{if}y the condition (v) of the same definition.

Notice that, for $x \in \G$,
\begin{align*}
    1_{Y_g}|_x = \begin{cases}
        1_{d(x)}, \text{ if } x \in X_g, \\ 
        0, \text{ otherwise}.
    \end{cases}
\end{align*}

So
\begin{align*}
    u_{r(g),g}|_x & = \begin{cases}
    \tilde{w}_{x^{-1},r(g)}\tilde{w}_{x^{-1}r(g),g}\tilde{w}_{x^{-1},r(g)g}^{-1}, \text{ if } x \in X_g, \\
    0, \text{ otherwise}
    \end{cases} \\
    & = \begin{cases}
    \tilde{w}_{x^{-1},r(x)}\tilde{w}_{x^{-1}r(x),g}\tilde{w}_{x^{-1},g}^{-1}, \text{ if } x \in X_g, \\
    0, \text{ otherwise}
    \end{cases} \\
    & = \begin{cases}
    \tilde{w}_{x^{-1},r(x)}\tilde{w}_{x^{-1}d(x^{-1}),g}\tilde{w}_{x^{-1},g}^{-1}, \text{ if } x \in X_g, \\
    0, \text{ otherwise}
    \end{cases} \\
    & = \begin{cases}
    \tilde{w}_{x^{-1},r(x)}\tilde{w}_{x^{-1},g}\tilde{w}_{x^{-1},g}^{-1}, \text{ if } x \in X_g, \\
    0, \text{ otherwise}
    \end{cases}
    \end{align*}

    \begin{align*}
    & = \begin{cases}
    \tilde{w}_{x^{-1},r(x)}, \text{ if } x \in X_g, \\ 
    0, \text{ otherwise}
    \end{cases} \\
    & = \begin{cases}
    \tilde{w}_{x^{-1},d(x^{-1})}, \text{ if } x \in X_g, \\
    0, \text{ otherwise}
    \end{cases} \\
    & = \begin{cases}
    1_{r(x^{-1})}, \text{ if } x \in X_g, \\
    0, \text{ otherwise}
    \end{cases} \\
    & = \begin{cases}
    1_{d(x)}, \text{ if } x \in X_g, \\ 
    0, \text{ otherwise}
    \end{cases} \\
    & = 1_{Y_g}|_x.
\end{align*}

On the other hand, 

\begin{align*}
    u_{g,d(g)}|_x & = \begin{cases}
    \tilde{w}_{x^{-1},g}\tilde{w}_{x^{-1}g,d(g)}\tilde{w}_{x^{-1},gd(g)}^{-1}, \text{ if } x \in X_g, \\
    0, \text{ otherwise}
    \end{cases} \\
    & = \begin{cases}
    \tilde{w}_{x^{-1},g}\tilde{w}_{x^{-1}g,d(x^{-1}g)}\tilde{w}_{x^{-1},g}^{-1}, \text{ if } x \in X_g, \\
    0, \text{ otherwise}
    \end{cases} \\
    & = \begin{cases}
    \tilde{w}_{x^{-1},g}1_{r(x^{-1}g)}\tilde{w}_{x^{-1},g}^{-1}, \text{ if } x \in X_g, \\ 
    0, \text{ otherwise}
    \end{cases} \\
    & = \begin{cases}
    \tilde{w}_{x^{-1},g}1_{r(x^{-1})}\tilde{w}_{x^{-1},g}^{-1}, \text{ if } x \in X_g, \\
    0, \text{ otherwise}
    \end{cases} \\
    & = \begin{cases}
    \tilde{w}_{x^{-1},g}\tilde{w}_{x^{-1},g}^{-1}, \text{ if } x \in X_g, \\ 
    0, \text{ otherwise}
    \end{cases} \\
    & = \begin{cases}
    1_{r(x^{-1})}, \text{ if } x \in X_g, \\ 
    0, \text{ otherwise}
    \end{cases} \\
    & = \begin{cases}
    1_{d(x)}, \text{ if } x \in X_g, \\ 
    0, \text{ otherwise}
    \end{cases} \\
    & = 1_{Y_g}|_x.
\end{align*}

Therefore $\beta$ is a groupoid twisted global action of $\G$ over $\mathcal{F} \cong \prod_{g \in \G} R_g$, where $R_g = R,$ $\forall g \in \G$.

Now, for each $e \in \G_0$, define $\varphi_e: D_e \rightarrow Y_e$ by 
\begin{align*}
    \varphi_e(a)|_h = \begin{cases}
    \alpha_{h^{-1}}(a1_h), \text{ if } h \in X_e, \\ 
    0, \text{ otherwise,}
    \end{cases}
\end{align*}
for all $a \in D_e$ and $h \in \G$. By the definition, it follows that $\varphi_e(a)|_e = a$. Thus, $\varphi_e$ is a monomorphism of rings, for all $e \in \G_0$.

Let $E_g$ be the subring of $Y_g$ generated by $\cup_{r(h) = r(g)}\beta_h(\varphi_{d(h)}(D_{d(h)})),$ for all $g \in \G$. Notice that $\varphi_{d(g)}(D_{d(g)}) \subseteq E_{d(g)}$. Let $T = \prod_{e \in \G_0} E_e$ and for each $e \in \G_0$, let $i_e: E_e \rightarrow T$ be the injective map given by $i_e(x) = (x_l)_{l \in \G_0}$, with $x_e = x$ and $x_l = 0$ for all $l \neq e$. We will identify $i_e(E_e)$ with $E_e$. We shall prove that the restriction of $\beta$ to $T$ is a globalization for $\alpha$. We will denote this restriction by the same symbol $\beta$.

To begin with, we will show condition (iv) of Definition \ref{def22}. For $g,h \in \G$ and $a \in D_{g^{-1}}$,
\begin{align*}
    \beta_g(\varphi_{d(g)}(a))|_h & = \begin{cases}
    \tilde{w}_{h^{-1},g}\varphi_{d(g)}(a)(g^{-1}h)\tilde{w}_{h^{-1},g}^{-1}, \text{ if } h \in X_g, \\
    0, \text{ otherwise}
    \end{cases} \\
    & = \begin{cases}
    \tilde{w}_{h^{-1},g}\alpha_{h^{-1}g}(a1_{g^{-1}h})\tilde{w}_{h^{-1},g}^{-1}, \text{ if } x \in X_g, \\
    0, \text{ otherwise}
    \end{cases} \\
    & = \begin{cases}
    w_{h^{-1},g}\alpha_{h^{-1}g}(a1_{g^{-1}h})w_{h^{-1},g}^{-1}, \text{ if } x \in X_g, \\ 
    0, \text{ otherwise.}
    \end{cases}
\end{align*}

On the other hand,
\begin{align*}
    \varphi_{r(g)}(\alpha_g(a))|_h & = \begin{cases} 
    \alpha_{h^{-1}}(\alpha_g(a)1_h), \text{ if } h \in X_g, \\
    0, \text{ otherwise}
    \end{cases} \\
    & = \begin{cases} 
    \alpha_{h^{-1}}(\alpha_g(a)1_g1_h), \text{ if } x \in X_g, \\ 
    0, \text{ otherwise}
    \end{cases} \\
    & = \begin{cases} 
    \alpha_{h^{-1}}(\alpha_g(a1_{g^{-1}}1_{g^{-1}h})), \text{ if } x \in X_g, \\ 
    0, \text{ otherwise.}
    \end{cases}
\end{align*}

By the condition (iv) of the Definition \ref{def21} for $\alpha$, it follows that $$\beta_g(\varphi_{d(g)}(a)) = \varphi_{r(g)}(\alpha_g(a)),$$that is, (iv) of the Definition \ref{def22} holds.

Next step is to prove (iii) of the Definition \ref{def22}, that means, for all $g \in \G$, $$\varphi_{r(g)}(D_g) = \varphi_{r(g)}(D_{r(g)}) \cap \beta_g(\varphi_{d(g)}(D_{d(g)})).$$ 

An element on the right hand side can be written as $\varphi_{r(g)}(a_{r(g)}) = \beta_g(\varphi_{d(g)}(b_{d(g)}),$ for some $a_{r(g)} \in D_{r(g)}$ and $b_{d(g)} \in D_{d(g)}$. Then, for each $h \in \G$, $\varphi_{r(g)}(a_{r(g)})|_h = \beta_g(\varphi_{d(g)}(b_{d(g)}))|_h$ implies that
\begin{align*}
    & \begin{cases}
     \alpha_{h^{-1}}(\alpha_g(a_{r(g)})1_h), \text{ if } h \in X_g, \\ 
    0, \text{ otherwise}
    \end{cases} \\
    = & \begin{cases}
    \tilde{w}_{h^{-1},g}\varphi_{d(g)}(b_{d(g)})\tilde{w}_{h^{-1},g}^{-1}, \text{ if } h \in X_g, \\
    0, \text{ otherwise}
    \end{cases} \\
    = & \begin{cases}
    \tilde{w}_{h^{-1},g}\alpha_{h^{-1}g}(b_{d(g)}1_{g^{-1}h})\tilde{w}_{h^{-1},g}^{-1}, \text{ if } h \in X_g, \\
    0, \text{ otherwise.}
    \end{cases}
\end{align*}

Thus $\alpha_{h^{-1}}(\alpha_g(a_{r(g)})1_h) = \tilde{w}_{h^{-1},g}\alpha_{h^{-1}g}(b_{d(g)}1_{g^{-1}h})\tilde{w}_{h^{-1},g}^{-1}$ if $h \in X_g$. Take $h = r(g)$. Therefore
$$\alpha_{r(g)}(\alpha_g(a_{r(g)})1_{r(g)}) = \tilde{w}_{r(g),g}\alpha_{r(g)g}(b_{d(g)}1_{g^{-1}r(g)})\tilde{w}_{r(g),g}^{-1},$$ which implies $a_{r(g)} = \alpha_{g}(b_{d(g)}1_{g^{-1}}) \in D_g$.

So $\varphi_{r(g)}(a_{r(g)}) \in \varphi_{r(g)}(D_g)$ and then $\varphi_{r(g)}(D_{r(g)}) \cap \beta_g(\varphi_{d(g)}(D_{d(g)})) \subseteq \varphi_{r(g)}(D_g).$

For the reverse inclusion, given an arbitrary $a_g \in D_g$, it follows that

\begin{align*}
    \beta_g(\varphi_{d(g)}(\alpha_g^{-1}(a_g)))|_h & = \begin{cases}
    \tilde{w}_{h^{-1},g}\varphi_{d(g)}(\alpha_g^{-1}(a_g))(g^{-1}h)\tilde{w}_{h^{-1},g}^{-1}, \text{ if } h \in X_g, \\
    0, \text{ otherwise}
    \end{cases}
    \end{align*}

    \begin{align*}
    & = \begin{cases}
    \tilde{w}_{h^{-1},g}\alpha_{h^{-1}g}(\alpha_g^{-1}(a_g)1_{g^{-1}h})\tilde{w}_{h^{-1},g}^{-1}, \text{ if } h \in X_g, \\ 
    0, \text{ otherwise}
    \end{cases} \\
    & = \begin{cases}
    \tilde{w}_{h^{-1},g}\alpha_{h^{-1}g}(\alpha_g^{-1}(a_g)1_{g^{-1}h}) \\
    \alpha_{h^{-1}g}(1_{g^{-1}}1_{g^{-1}h})\tilde{w}_{h^{-1},g}^{-1}, \text{ if } h \in X_g, \\
    0, \text{ otherwise}
    \end{cases} \\
    & = \begin{cases}
    \tilde{w}_{h^{-1},g}\alpha_{h^{-1}g}(\alpha_g^{-1}(a_g)1_{g^{-1}h}) \\ 
    1_{h^{-1}g}1_{h^{-1}}\tilde{w}_{h^{-1},g}^{-1}, \text{ if } h \in X_g, \\
    0, \text{ otherwise,}
    \end{cases}
\end{align*}
where the last equality follows by (iii) of the Definition \ref{def21}. Moreover, using the hyphotesis $\tilde{w}_{g,h}1_g1_{gh} = w_{g,h}$ and (iv) from the Definition \ref{def21}, we have that 
\begin{align*}
    \beta_g(\varphi_{d(g)}(\alpha_g^{-1}(a_g)))|_h & = \begin{cases}
    w_{h^{-1},g}\alpha_{h^{-1}g}(\alpha_g^{-1}(a_g)1_{g^{-1}h})w_{h^{-1},g}^{-1}, \text{ if } h \in X_g, \\ 
    0, \text{ otherwise}
    \end{cases} \\
    & = \begin{cases}
    \alpha_{h^{-1}} \circ \alpha_g(\alpha_g^{-1}(a_g)1_{g^{-1}}1_{g^{-1}h}), \text{ if } h \in X_g, \\
    0, \text{ otherwise}
    \end{cases} \\
    & = \begin{cases}
    \alpha_{h^{-1}}(a_g\alpha_g(1_{g^{-1}}1_{g^{-1}h})), \text{ if } h \in X_g, \\ 
    0, \text{ otherwise}
    \end{cases} \\
    & = \begin{cases}
    \alpha_{h^{-1}}(a_g1_{g}1_{gh})), \text{ if } h \in X_g, \\ 
    0, \text{ otherwise}
    \end{cases} \\
    & = \begin{cases}
    \alpha_{h^{-1}}(a_g1_{h})), \text{ if } h \in X_g, \\ 
    0, \text{ otherwise}
    \end{cases} \\
    & = \varphi_{r(g)}(a_g)|_h.
\end{align*}

Then $\varphi_{r(g)}(D_g) \subseteq \varphi_{r(g)}(D_{r(g)}) \cap \beta_g(\varphi_{d(g)}(D_{d(g)}))$ and (ii) of the Definition \ref{def22} is valid.

Next we check (i) of the Definition \ref{def21}. Since $E_g (= E_{r(g)})$ is the subring of $Y_g$ generated by $\cup_{r(h) = r(g)}\beta_h(\varphi_{d(h)}(D_{d(h)})),$ to see that $\varphi_e(D_e)$ is an ideal of $E_e$ for each $e \in \G_0$ is enough to show that $\beta_h(\varphi_{d(h)}(a_{d(h)})) \cdot \varphi_{r(g)}(b_{r(g)})$, $\varphi_{r(g)}(b_{r(g)}) \cdot  \beta_h(\varphi_{d(h)}(a_{d(h)})) \in \varphi_{r(g)}(D_{r(g)})$, for all $g \in \G$, $h \in X_g$, $a_{d(h)} \in D_{d(h)}$ and $b_{r(g)} \in D_{r(g)}$.

For $k \in \G$, using a similar argument to what was done to show (iii) of the Definition \ref{def22}, we have

\begin{align*}
    \beta_h(\varphi_{d(h)}(a_{d(h)}))|_k \cdot \varphi_{r(g)}(b_{r(g)})|_k & = \begin{cases}
    \tilde{w}_{k^{-1},h}\varphi_{d(h)}(a_{d(h)})(h^{-1}k)\tilde{w}_{h^{-1},g}^{-1} \cdot \alpha_{k^{-1}}(b_{r(g)}1_k), \text{ if } k \in X_g, \\ 
    0, \text{ otherwise}
    \end{cases} \\
    & = \begin{cases}
    \tilde{w}_{k^{-1},h}\alpha_{k^{-1}h}(a_{d(h)}1_{h^{-1}k})\tilde{w}_{h^{-1},g}^{-1}\cdot \alpha_{k^{-1}}(b_{r(g)}1_k), \text{ if } k \in X_g, \\ 
    0, \text{ otherwise}
    \end{cases} \\
    & = \begin{cases}
    \alpha_{k^{-1}} \circ \alpha_{h}(a_{d(h)}1_{h^{-1}}1_{h^{-1}k}) \cdot \alpha_{k^{-1}}(b_{r(g)}1_k), \text{ if } k \in X_g \\ 
    0, \text{ otherwise}
    \end{cases} \\
    & = \begin{cases}
    \alpha_{k^{-1}}(\alpha_{h}(a_{d(h)}1_{h^{-1}})b_{r(g)}1_k), \text{ if } k \in X_g, \\ 
    0, \text{ otherwise}
    \end{cases} \\
    & = \varphi_{r(g)}(\alpha_h(a_{d(h)}1_{h^{-1}})b_{r(g)})|_k.
\end{align*}

This yields that $\beta_h(\varphi_{d(h)}(a_{d(h)})) \cdot \varphi_{r(g)}(b_{r(g)}) \in \varphi_{r(g)}(D_{r(g)})$. Similarly, $$\varphi_{r(g)}(b_{r(g)}) \cdot \beta_h(\varphi_{d(h)}(a_{d(h)})) \in \varphi_{r(g)}(D_{r(g)}).$$

We show next that $u$ satisfies (v) of the Definition \ref{def22}. By the conditions (vi) and (iii) of the Definition \ref{def21}, we have
\begin{align*}
    \varphi_{r(g)}(w_{g,h})|_k & = \begin{cases}
    \alpha_{k^{-1}}(w_{g,h}1_k), \text{ if } k \in X_g, \\ 
    0, \text{ otherwise}
    \end{cases} \\
    & = \begin{cases}
    w_{k^{-1},g}w_{k^{-1}g,h}w_{k^{-1},gh}^{-1}, \text{ if } k \in X_g, \\
    0, \text{ otherwise}
    \end{cases} \\
    & = \begin{cases}
    1_{k^{-1}}1_{k^{-1}g}1_{k^{-1}gh}\tilde{w}_{k^{-1},g}\tilde{w}_{k^{-1}g,h}\tilde{w}_{k^{-1},gh}^{-1}, \text{ if } k \in X_g, \\
    0, \text{ otherwise}
    \end{cases} \\
    & = \begin{cases}
    \alpha_{k^{-1}}(1_{k}1_g1_{gh})\tilde{w}_{k^{-1},g}\tilde{w}_{k^{-1}g,h}\tilde{w}_{k^{-1},gh}^{-1}, \text{ if } k \in X_g, \\
    0, \text{ otherwise}
    \end{cases} \\
    & = \varphi_{r(g)}(1_g1_{gh})|_k\cdot u_{g,h}|_k = (\varphi_{r(g)}(1_g1_{gh}) \cdot u_{g,h})|_k.
\end{align*}

Thus $\varphi_{r(g)}(w_{g,h}) = \varphi_{r(g)}(1_g1_{gh}) \cdot u_{g,h}$. For $a \in D_gD_{gh}$,
\begin{align*}
    \varphi_{r(g)}(aw_{g,h}) & = \varphi_{r(g)}(a)\varphi_{r(g)}(w_{g,h}) = \varphi_{r(g)}(a)\varphi_{r(g)}(1_g1_{gh})u_{g,h} = \varphi_{r(g)}(a1_g1_{gh})u_{g,h} = \varphi_{r(g)}(a)u_{g,h}.
\end{align*}

By a similar reasoning, we have that $\varphi_{r(g)}(w_{g,h}a) = u_{g,h}\varphi_{r(g)}(a)$.

Observe that $E_g = \sum_{r(h) = r(g)} \beta_h(\varphi_{d(h)}(D_{d(h)}))$ for all $g \in \G$. This is a consequence of the fact that each $\varphi_{d(h)}(D_{d(h)})$ is an ideal of $E_{d(h)}$, as it was showed previously.

It remains to prove that each $u_{g,h} \in \mathcal{U}(E_g)$. Recall that, a priori, $u_{g,h} \in \mathcal{U}(Y_g)$. Let's check that $u_{g,h}E_g = E_g = E_gu_{g,h}$.

\noindent \textbf{Step 1:} $u_{g,h}\varphi_{r(g)}(a_g) = \varphi_{r(g)}(\tilde{w}_{g,h}a_g),$ for all $(g,h) \in \G^{(2)}$ and $a_g \in E_g$.

In fact, using the hypothesis $(\ast)$, for all $t \in \G$, we have that 
\begin{align*}
    u_{g,h}\varphi_{r(g)}(a_g)|_k & = \begin{cases}
    \tilde{w}_{k^{-1},g}\tilde{w}_{k^{-1}g,h}\tilde{w}_{k^{-1},gh}^{-1}\alpha_{k^{-1}}(a_g1_k), \text{ if } k \in X_g, \\
    0, \text{ otherwise}
    \end{cases} \\
    & = \begin{cases}
    \alpha_{k^{-1}}(\tilde{w}_{g,h}1_k)\alpha_{k^{-1}}(a_g1_k), \text{ if } k \in X_g, \\ 
    0, \text{ otherwise}
    \end{cases} \\
    & = \begin{cases}
    \alpha_{k^{-1}}(\tilde{w}_{g,h}a_g1_k) \text{ if } k \in X_g, \\ 
    0, \text{ otherwise}
    \end{cases} \\
    & = \varphi_{r(g)}(\tilde{w}_{g,h}a_g)|_k.
\end{align*}

By Step 1, it follows that $u_{g,h}\varphi_{r(g)}(a_g) \in \varphi_{r(g)}(D_{r(g)})$. It is analogously seen that $u_{g,h}^{-1}\varphi_{r(g)}(a_g) = \varphi_{r(g)}(\tilde{w}_{g,h}^{-1}a_g)$, so that $u_{g,h}^{-1}\varphi_{r(g)}(a_g) \in \varphi_{r(g)}(D_{r(g)})$.

\noindent \textbf{Step 2:} $\beta_k^{-1}(f) = u_{k^{-1}, k}^{-1}\beta_{k^{-1}}(f)u_{k^{-1},k}$, for all $k \in \G$ and $f \in Y_k$.

Take $h \in \G$.
\small{\begin{align*}
    u_{k^{-1}, k}^{-1}\beta_{k^{-1}}(f)u_{k^{-1},k}|_h & = \begin{cases}
    \tilde{w}_{h^{-1},k^{-1}k}\tilde{w}_{h^{-1}k^{-1},k}^{-1}\tilde{w}_{h^{-1},k^{-1}}^{-1}\tilde{w}_{h^{-1},k^{-1}}f(kh) \\
    \cdot \; \tilde{w}_{h^{-1},k^{-1}}^{-1}\tilde{w}_{h^{-1},k^{-1}}\tilde{w}_{h^{-1}k^{-1},k}\tilde{w}_{h^{-1},k^{-1}k}^{-1}, \text{ if } h \in X_k, \\
    0, \text{ otherwise}
    \end{cases} \\
    & = \begin{cases}
    \tilde{w}_{h^{-1}k^{-1},k}^{-1}f(kh)\tilde{w}_{h^{-1}k^{-1},k}, \text{ if } h \in X_k, \\ 
    0, \text{ otherwise}
    \end{cases} \\
    & = \beta_k^{-1}(f)|_h.
\end{align*}}

\noindent \textbf{Step 3:} $\beta_g(u_{h,k})u_{g,hk} = u_{g,h}u_{gh,k}$.

By the hypothesis $(\ast)$, for $x \in \G$,
\begin{align*}
    & (\beta_g(u_{h,k})u_{g,hk})|_x \\
    & = \begin{cases}
    [\tilde{w}_{x^{-1},g}u_{h,k}(g^{-1}x)\tilde{w}_{x^{-1},g}^{-1}][\tilde{w}_{x^{-1},g}\tilde{w}_{x^{-1}g,hk}\tilde{w}_{x^{-1},ghk}^{-1}], \text{ if } x \in X_g, \\ 
    0, \text{ otherwise}
    \end{cases} \\
    & = \begin{cases}
    \tilde{w}_{x^{-1},g}\tilde{w}_{x^{-1}g,h}\tilde{w}_{x^{-1}gh,k}\tilde{w}_{x^{-1}g,hk}^{-1}\tilde{w}_{x^{-1},g}^{-1}\tilde{w}_{x^{-1},g}\tilde{w}_{x^{-1}g,hk}\tilde{w}_{x^{-1},ghk}^{-1}, \text{ if } x \in X_g, \\
    0, \text{ otherwise}
    \end{cases} \\
    & = \begin{cases}
    \tilde{w}_{x^{-1},g}\tilde{w}_{x^{-1}g,h}\tilde{w}_{x^{-1}gh,k}\tilde{w}_{x^{-1},ghk}^{-1}, \text{ if } x \in X_g, \\ 
    0, \text{ otherwise}
    \end{cases} \\
    & = \begin{cases}
    [\tilde{w}_{x^{-1},g}\tilde{w}_{x^{-1}g,h}\tilde{w}_{x^{-1},gh}^{-1}][\tilde{w}_{x^{-1},gh}\tilde{w}_{x^{-1}gh,k}\tilde{w}_{x^{-1},ghk}^{-1}], \text{ if } x \in X_g, \\ 
    0, \text{ otherwise}
    \end{cases} \\
    & = (u_{g,h}u_{gh,k})|_x.
\end{align*}

Using Steps 1, 2 and 3, we have that, for all $a_g \in D_{r(g)}$,
\begin{align*}
    \beta_k^{-1}(u_{g,h})\varphi_{r(g)}(a_g) & = u_{k^{-1},k}^{-1}\beta_{k^{-1}}(u_{g,h})u_{k^{-1},k}\varphi_{r(g)}(a_g) \\
    & = u_{k^{-1},k}^{-1}u_{k^{-1},g}u_{k^{-1}g,h}u_{k^{-1},gh}^{-1}u_{k^{-1},k}\varphi_{r(g)}(a_g) \in \varphi_{r(g)}(D_{r(g)}).
\end{align*}

Therefore, $\beta_k^{-1}(u_{g,h})\varphi_{r(g)}(D_{r(g)}) \subseteq \varphi_{r(g)}(D_{r(g)})$. Then $u_{g,h}\beta_k(\varphi_{r(g)}(D_{r(g)})) \subseteq \beta_k(\varphi_{r(g)}(D_{r(g)})),$ concluding that $u_{g,h}E_g \subseteq E_g$.

Similarly, $u_{g,h}^{-1}E_g \subseteq E_g$. An analogous argument gives $E_gu_{g,h}, E_gu_{g,h}^{-1} \subseteq E_g$, which shows that $u_{g,h}E_g = E_g = E_gu_{g,h}$.
\end{proof}

\begin{exe}
Consider $R$, $\G$ and $\alpha$ as in Example \ref{ex1}. We have that $\alpha$ is a globalizable groupoid twisted partial action by taking
\begin{align*}
    \tilde{w}_{r(g),r(g)} = \tilde{w}_{r(g),g} & = \tilde{w}_{g,d(g)}  = 1_{r(g)} = e_1 + e_2 \\
    \tilde{w}_{d(g),d(g)} = \tilde{w}_{d(g),g^{-1}} & = \tilde{w}_{g^{-1},r(g)} = 1_{d(g)} = e_3 + e_4 \\
    \tilde{w}_{g,g^{-1}}  = -1_{r(g)} =  -(e_1 &+ e_2), \; \tilde{w}_{g^{-1},g} = -1_{d(g)} = -(e_3 + e_4)
\end{align*}
and applying Theorem \ref{teoglobal}.
\end{exe}

\section{Twisted Crossed Products and Morita Equivalence}

In this section, given a twisted partial action of a groupoid, we introduce the associated twisted crossed product and show that the crossed products of the partial action and of its globalization are Morita equivalent.

\begin{defi}
Let $\alpha = (\{D_g\}_{g \in \mathcal{G}}, \{\alpha_g\}_{g \in \mathcal{G}}, \{w_{g,h}\}_{(g,h) \in \mathcal{G}^{(2)}})$ be a groupoid twisted partial action of $\G$ on a ring $R$. We define the \emph{twisted crossed product} $R *_{\alpha,w} G$ as
\begin{align*}
    R *_{\alpha,w} \G = \left \{ \sum^{\text{finite}}_{g \in \G} a_g \delta_g : a_g \in D_g \right \} = \bigoplus_{g \in \G} D_g \delta_g,
\end{align*}
where the $\delta_g$'s are symbols. The addition is the usual and the product is given by
\begin{align*}
    (a_g \delta_g)(b_h \delta_h) = \begin{cases} 
    \alpha_g(\alpha_{g}^{-1}(a_g)b_h)w_{g,h} \delta_{gh}, \text{ if } (g,h) \in \G^{(2)}, \\
    0, \text{ otherwise.}
    \end{cases}
\end{align*}
\end{defi}

\begin{theorem} \label{teoassoc}
    Let $\alpha = (\{D_g\}_{g \in \mathcal{G}}, \{\alpha_g\}_{g \in \mathcal{G}}, \{w_{g,h}\}_{(g,h) \in \mathcal{G}^{(2)}})$ be a groupoid twisted partial action of $\G$ on $R$. Then the crossed product $R *_{\alpha,w} \G$ is associative.    
\end{theorem}
\begin{proof}
    Let $a_g\delta_g, b_h\delta_h, c_k\delta_k \in R *_{\alpha,w} \G$. On the one hand,
    \begin{align*}
        [(a_g\delta_g)(b_h\delta_h)](c_k\delta_k) & = \begin{cases}
            (\alpha_g(\alpha_g^{-1}(a_g)b_h)w_{g,h}\delta_{gh})(c_k\delta_k), \text{ if } (g,h) \in \G^{(2)}, \\
            0, \text{ otherwise}
        \end{cases} \\
        & = \begin{cases}
            \alpha_{gh}(\alpha_{gh^{-1}}(\alpha_g(\alpha_g^{-1}(a_g)b_h)w_{g,h})c_k)w_{gh,k}\delta_{ghk}, \text{ if } (g,h,k) \in \G^{(3)}, \\
            0, \text{ otherwise}.
        \end{cases}
    \end{align*}
    
    On the other hand,
    \begin{align*}
        (a_g\delta_g)[(b_h\delta_h)(c_k\delta_k)] & = \begin{cases}
        (a_g\delta_g)(\alpha_h(\alpha_h^{-1}(b_h)c_k)w_{h,k}\delta_{hk}), \text{ if } (h,k) \in \G^{(2)}, \\
        0, \text{ otherwise}
        \end{cases} \\
        & = \begin{cases}
        \alpha_g(\alpha_{g}^{-1}(a_g)\alpha_h(\alpha_h^{-1}(b_h)c_k)w_{h,k}) \\ 
        w_{g,hk}\delta_{ghk}, \text{ if } (g,h,k) \in \G^{(3)}, \\
        0, \text{ otherwise}.
        \end{cases}
    \end{align*}
    
   Thus, we only need to consider the case where $(g,h,k) \in \G^{(3)}$. The result then follows from \cite[Theorem 2.4]{dokuchaev2008crossed}.
\end{proof}

Let $\alpha = (\{D_g\}_{g \in \mathcal{G}}, \{\alpha_g\}_{g \in \mathcal{G}}, \{w_{g,h}\}_{(g,h) \in \mathcal{G}^{(2)}})$ be a twisted partial action of a groupoid $\G$ on a ring $R$, and let $\beta = (E_g, \beta_g, u{g,h})_{(g,h) \in \G^{(2)}}$ be a twisted global action of $\G$ on a ring $T$ such that $\beta$ is a globalization of $\alpha$. Denote by $A = R *_{\alpha,w} \G$ and $B = T *_{\beta,u} \G$ the corresponding crossed products. For simplicity, we identify $\varphi_e(D_e)$ with $D_e$ for all $e \in \G_0$, where the maps $\varphi_e$ are the monomorphisms introduced in Definition \ref{def22}.

Suppose that $\G_0$ is finite. Then $A$ has unity $1_A = \sum_{g \in \G_0} 1_e\delta_e$.

\begin{prop} \label{propACB} Keeping the notations above, it follows that:
    \begin{enumerate}
        \item[(i)] $B1_A = \left \{ \sum_{g \in \G}^{\text{finite}} c_g\delta_g : c_g \in \beta_g(D_{d(g)}) \right \}.$
        
        \item[(ii)] $1_AB = \left \{ \sum_{g \in \G}^{\text{finite}} c_g\delta_g : c_g \in D_{r(g)} \right \}.$
        
        \item[(iii)] $1_AB1_A = A$.
        
        \item[(iv)] $B1_AB = B.$
    \end{enumerate}
\end{prop}
\begin{proof}

(i): For all $g \in \G$, $s \in E_g$ there is $t \in E_{g^{-1}}$ such that $s = \beta_g(t)$. Hence,
\begin{align*}
    (s\delta_g)1_A = (s\delta_g)(1_{d(g)}\delta_{d(g)}) = \beta_g(\beta_g^{-1}(s)1_{d(g)})u_{g,d(g))}\delta_{gd(g)} = \beta_g(\beta_g^{-1}(s)1_{d(g)})\delta_g = \beta_g(t1_{d(g)})\delta_g,
\end{align*}
where $c_g := \beta_g(t1_{d(g)}) \in \beta_g(D_{d(g)})$.

For the reverse inclusion, take $a \in D_{d(g)}$ and $c_g = \beta_g(a)$. Then
\begin{align*}
    c_g\delta_g = \beta_g(a)\delta_g & = \beta_g(a1_{d(g)})\delta_g = \beta_g(a)\beta_g(1_{d(g)})\delta_g = c_g\beta_g(1_{d(g)})\delta_g \\
    & = \beta_g(\beta_g^{-1}(c_g)1_{d(g)})u_{g,d(g)}\delta_{gd(g)} = (c_g\delta_g)(1_{d(g)}\delta_{d(g)}) = (c_g\delta_g)1_A.
\end{align*}

(ii): Let $s \in E_g = E_{r(g)}$. We have that
\begin{align*}
    1_A(s\delta_g) = (1_{r(g)}\delta_{r(g)})(s \delta_g) & = \beta_{r(g)}(\beta_{r(g)}^{-1}(1_{r(g)})s)u_{r(g),g}\delta_{r(g)g} = 1_{r(g)}\beta_{r(g)}(s)\delta_g = 1_{r(g)}s\delta_g,
\end{align*}
and since $D_{r(g)}$ is an ideal of $E_{r(g)}$, $1_{r(g)}s \in D_{r(g)}$.

For the reverse inclusion, let $c_g \in D_{r(g)}$. Then $c_g\delta_g = (1_{r(g)}\delta_{r(g)})(c_g\delta_g) = 1_A(c_g\delta_g)$.

(iii): The inclusion $A \subseteq 1_AB1_A$ is clear. Let $a \in E_g = E_{r(g)}$. Then $a1_g \in E_{r(g)}1_g = D(g)$. Thus
\begin{align*}
    1_A(a\delta_g)1_A & = (1_r(g)\delta_{r(g)})(a\delta_g)(1_{d(g)}\delta_{d(g)}) = (1_{r(g)}a\delta_g)(1_{d(g)}\delta_{d(g)}) = \beta_g(\beta_g^{-1}(1_{r(g)}a)1_{d(g)})u_{g,d(g)}\delta_g.
\end{align*}

Since $1_{d(g)} = 1_{r(g^{-1})} \in E_{g^{-1}}$, we have that
\begin{align*}
    1_A(a\delta_g)1_A = 1_{r(g)}a\beta_g(1_{d(g)})\delta_g = a1_{r(g)}\beta_g(1_{d(g)})\delta_g = a1_g\delta_g \in A.
\end{align*}

(iv): We just need to show that $B \subseteq B1_AB$. Since $E_h = \sum_{r(g) = r(h)} \beta_g(D_{d(g)})$, for all $h \in \G$, so the result follows from (i) and (ii). In fact, for all $a \in D_{d(g)}$ with $r(g) = r(h)$,
\begin{align*}
    \beta_g(a)\delta_h \overset{(*)}{=} (\beta_g(a)\delta_g)(1_{d(g)}u_{g^{-1},g}^{-1}u_{g^{-1},h}\delta_{g^{-1}h}) \in (B1_A)B = B1_AB.
\end{align*}

The equality $(*)$ is valid because
\begin{align*}
    & (\beta_g(a)\delta_g)(1_{d(g)}u_{g^{-1},g}^{-1}u_{g^{-1},h}\delta_{g^{-1}h}) = \beta_g(\beta_g^{-1}(\beta_g(a))1_{d(g)}u_{g^{-1},g}^{-1}u_{g^{-1},h})u_{g,g^{-1}h}\delta_{g(g^{-1}h)} \\
    & = \beta_g(a1_{d(g)}u_{g^{-1},g}^{-1}u_{g^{-1},h})u_{g,g^{-1}h}\delta_h = \beta_g(au_{g^{-1},g}^{-1})u_{g,g^{-1}}u_{r(h),h}\delta_h \\
    & \overset{(**)}{=} \beta_g(a)u_{g,g^{-1}}^{-1}u_{g,g^{-1}}\delta_h= \beta_g(a)\delta_h,
\end{align*}
where $(**)$ follows from Definition \ref{def21} taking $h = g^{-1}, k = g$ and $a = aw_{g^{-1},g}^{-1}$.
\end{proof}

A Morita context is a six-tuple ($R,R',M,M',\tau,\tau'$) where $R,R'$ are rings, $M$ is a $R,R'$-bimodule, $M'$ is a $R',R$-bimodule, $\tau : M \otimes_{R'} M' \to R$, $\tau' : M' \otimes_{R} M \to R'$ are bimodule maps such that
\begin{align*}
    \tau(x \otimes x')y = x\tau'(x' \otimes y), & \text{ for all } x,y \in M, x' \in M'
\end{align*}
and
\begin{align*}
    \tau'(x' \otimes x)y' = x'\tau(x \otimes y'), & \text{ for all } x',y' \in M', x \in M.
\end{align*}

 It follows from \cite[Theorems 4.1.4 and 4.1.17]{rowen2012ring} that if $\tau$ and $\tau'$ are onto, then the categories of $R$-modules and $R'$-modules are equivalent. In this case, we say that $R$ and $R'$ are \emph{Morita equivalent}.

%Now we show that when $\alpha$ is a globalizable groupoid twisted partial action of $\G$ on $R$ with globalization $\beta$ of $\G$ on $T$, then the rings $R *_{\alpha,w} \G$ and $T *_{\beta,u} \G$ are Morita equivalent with $\tau$ and $\tau'$ onto. 
The next result generalizes \cite[Theorem 3.2]{bagio} and \cite[Theorem 3.1]{dokues}.

\begin{theorem}\label{morita}
    Let $\alpha$ be a globalizable groupoid twisted partial action of a finite groupoid $\G$ on a unital ring $R$ and $\beta$ be a globalization of $\alpha$ of $\G$ on $T$. Then the rings  $A = R *_{\alpha,w} \G$ and $B = T *_{\beta,u} \G$ are Morita equivalent.
\end{theorem}
\begin{proof}
    Since $\G$ is finite, $R$ is unital and $\alpha$ is globalizable, it follows that $R$, $A$ and $B$ are unital rings. Set $M = \bigoplus_{g \in \G} D_{r(g)}\delta_g \subseteq B$ and $N = \bigoplus_{g \in \G} \beta_g(D_{d(g)})\delta_g \subseteq B$.  We first show that $M$ is an $(A,B)$-bimodule and that $N$ is a $(B,A)$-bimodule.
    
    \noindent\textbf{Assertion 1:} $M$ is a right ideal of $B$ and $N$ is a left ideal of $B$.
    
    Let $m\delta_g \in M$ and $b\delta_h \in B$. Then
    \begin{align*}
        (m\delta_g)(b\delta_h) = \begin{cases}
            m\beta_g(b)u_{g,h}\delta_{gh}, \text{ if } (g,h) \in \G^{(2)} \\
            0, \text{ otherwise}.
        \end{cases}
    \end{align*}
    
    Since $b \in E_h = E_{r(h)} = E_{d(g)} = E_{g^{-1}}$, it follows that $\beta_g(b) \in E_g$. Moreover, since $D_{r(g)}$ is an ideal of $E_{r(g)} = E_g$, we have $m\beta_g(b) \in D_{r(g)}$. Furthermore, because $E_{r(g)} = E_g = E_{gh}$, we obtain $D_{r(g)}E_g \subseteq D_gD_{gh}$, hence $m\beta_g(b)u_{g,h} = m\beta_g(b)w_{g,h} \in D_{g}D_{gh} \subseteq D_{r(g)}$. Therefore, $M$ is a right ideal of $B$.
    
    Let $b\delta_g \in B$ and $n\delta_h \in N$. Since $n = \beta_h(n')$ for some $n' \in D_{d(h)}$, we have
    \begin{align*}
        (b\delta_g)(n\delta_h) & = \begin{cases}
            b\beta_g(n)u_{g,h}\delta_{gh}, \text{ if } (g,h) \in \G^{(2)}, \\
            0, \text{ otherwise}
        \end{cases} \\
        & = \begin{cases}
            b\beta_g(\beta_h(n'))u_{g,h}\delta_{gh}, \text{ if } (g,h) \in \G^{(2)}, \\
            0, \text{ otherwise}
        \end{cases} \\
        & = \begin{cases}
        bu_{g,h}\beta_{gh}(n')\delta_{gh}, \text{ if } (g,h) \in \G^{(2)}, \\
        0, \text{ otherwise}.
        \end{cases}
    \end{align*}
    
    Now, $\beta_{gh}(D_{d(h)})$ is an ideal of $E_{d(h)} = E_h$, and hence $bu_{g,h}\beta_{gh}(n')$ $\in \beta_{gh}(D_{d(h)}) = \beta_{gh}(D_{d(gh)})$.
    
    \noindent\textbf{Assertion 2:} $AM \subseteq M$ and $NA \subseteq N$. 
    
    First, note that the products $AM$ and $NA$ are well defined by Proposition \ref{propACB}. Let $a\delta_g \in A$ and $m\delta_h \in M$. Then
    \begin{align*}
        (a\delta_g)(m\delta_h) = \begin{cases}
        a\beta_g(m)u_{g,h}\delta_{gh}, \text{ if } (g,h) \in \G^{(2)}, \\
        0, \text{ otherwise,}
        \end{cases}
    \end{align*}
    and this element belongs to $M$, since $D_{r(g)}$ is an ideal of $M$.
    
    Now, let $n\delta_g \in N$ and $a\delta_h \in A$. Since $n = \beta_g(n')$ for some $n' \in D_{d(g)}$, we have
    \begin{align*}
        (n\delta_g)(a\delta_h) & = \begin{cases}
            n\beta_g(a)u_{g,h}\delta_{gh}, \text{ if } (g,h) \in \G^{(2)}, \\
            0, \text{ otherwise}
        \end{cases}  = \begin{cases}
            \beta_g(n'a)u_{g,h}\delta_{gh}, \text{ if } (g,h) \in \G^{(2)}, \\
            0, \text{ otherwise}.
        \end{cases} 
    \end{align*}
    
    Since $a \in D_h \triangleleft D_{r(h)} = D_{d(g)}$ and $n' \in D_{d(g)}$, it follows that $a n' \in D_h$. Hence, there exists $x \in D_{h^{-1}}$ such that $a n' = \alpha_h(x)$. In this way,
    \begin{align*}
        \beta_g(n'a)u_{g,h}\delta_{gh} = \beta_{g}(\alpha_h(x))u_{g,h}\delta_{gh} = \beta_g(\beta_h(x))u_{g,h}\delta_{gh} = u_{g,h}\beta_{gh}(x)\delta_{gh} \in N.
    \end{align*}
    
    Assertions 1 and 2 ensure that $M$ is an $(A,B)$-bimodule and that $N$ is a $(B,A)$-bimodule. We now define
    \begin{align*}
        \tau : M \otimes_B N & \to A \\
        m \otimes n & \mapsto mn
    \end{align*}
    and 
    \begin{align*}
        \tau' : N \otimes_A M & \to B \\
        n \otimes m & \mapsto nm.
    \end{align*}
    
    It is easy to see that $\tau$ and $\tau'$ are bimodule maps and onto. Furthermore,
    \begin{align*}
        \tau(m \otimes n)m' = (mn)m' = m(nm') = m\tau'(n \otimes m')
    \end{align*}
    and
    \begin{align*}
        \tau'(n \otimes m)n' = (nm)n' = n(mn') = n\tau(m \otimes n'),
    \end{align*}
    for all $m,m' \in M$ and $n,n' \in N$ by Theorem \ref{teoassoc}. Therefore, $A$ and $B$ are Morita equivalent. \qedhere

\end{proof}

\section{Partial Projective Representations}
In this section, we extend the notion of partial projective representations introduced in \cite{dokuchaev2010partialprojectiverep,dokuchaev2012partialprojectiverep2} to the groupoid setting and study the structure of the associated partial Schur multiplier.

Let $\mathbb{K}$ be a field and let $\mathbb{K}^*$ denote its multiplicative group.

\begin{defi}
Let $\mathcal{S}$ be an inverse semigroupoid and $\mathcal{T}$ a semigroupoid. A map $\varphi : \mathcal{S} \to \mathcal{T}$ is said a \emph{partial homomorphism} if $(s,t) \in \mathcal{S}^{(2)}$ implies $\varphi(s)\varphi(t) \in \mathcal{T}^{(2)}$ and in this case
\begin{align*}
    \varphi(s^{-1})\varphi(s)\varphi(t) & = \varphi(s^{-1})\varphi(st), \\
    \varphi(s)\varphi(t)\varphi(t^{-1}) & = \varphi(st)\varphi(t^{-1}).
\end{align*}
\end{defi}

\begin{exe}
Consider $\G$ a groupoid and $\mathcal{E}(\G)$ the Exel's inverse category constructed in \cite{lata2020inverse} with generators $\{ [g] : g \in \G \}$ and relations
\begin{align*}
    ([g],[h]) \in \mathcal{E}(\G)^{(2)} \iff (g,h) & \in \G^{(2)}, \; [g^{-1}][g][h] = [g^{-1}][gh], \\
    [g][h][h^{-1}] = [gh][h^{-1}], & \; [r(g)][g] = [g] = [g][d(g)].
\end{align*}

The inclusion $f :\G \to \mathcal{E}(\G)$ given by $g \mapsto [g]$ is a partial homomorphism.
\end{exe}

Consider the semigroup $\mathcal{M}_n(\mathbb{K})$ of $n \times n$ matrices over $\mathbb{K}$, endowed with the usual matrix multiplication. The notion of projective representations of semigroups \cite{dokuchaev2010partialprojectiverep} is closely related to matrix monoids; however, we shall adopt a more general definition.

\begin{defi}
A \emph{$\mathbb{K}$-monoid} is a monoid $M$ with zero element $0_M$, endowed with a scalar multiplication map $\mathbb{K} \times M \to M$ such that, for all $x,y \in M$ and $a,b \in \mathbb{K}$,
\[
a(bx) = (ab)x, \qquad 1_{\mathbb{K}}x = x, \qquad a(xy) = (ax)y = x(ay),
\]
and, in addition,
\[
x 0_{\mathbb{K}} = 0_{\mathbb{K}} x = 0_M, \quad \text{for all } x \in M.
\]
We say that $M$ is a \emph{$\mathbb{K}$-cancellative monoid} if it is a $\mathbb{K}$-monoid and $ax = bx$ implies $a=b$ for every $0_M \neq x \in M$ and $a,b \in \mathbb{K}$.
\end{defi}

If $M$ is a $\mathbb{K}$-monoid, we define a relation $\lambda$ on $M$ by
\[
x \,\lambda\, y \iff x = ay \quad \text{for some } a \in \mathbb{K}^* .
\]
It is straightforward to verify that $\lambda$ is an equivalence relation and a congruence on $M$, and hence the quotient semigroup $\mathcal{P}(M) := M/\lambda$ is well defined.

Let $\Delta \colon \mathcal{S} \to \mathcal{P}(M)$ be a semigroupoid homomorphism and let $\xi \colon M \to \mathcal{P}(M)$ denote the canonical projection. Let $\xi' : \mathcal{P}(M) \to M$ be a choice of representatives. Clearly $\xi'$ is a right inverse to $\xi$. Define $\Gamma = \xi'\Delta$. Then $\Delta = \xi\xi'\Delta = \xi\Gamma$.

Since both $\Delta$ and $\xi$ are homomorphisms, for any composable pair $(x,y) \in \mathcal{S}^{(2)}$ we have
\[
\xi\big(\Gamma(xy)\big)
= \Delta(xy)
= \Delta(x)\Delta(y)
= \xi\big(\Gamma(x)\big)\,\xi\big(\Gamma(y)\big)
= \xi\big(\Gamma(x)\Gamma(y)\big).
\]
It follows that $\Gamma(xy)$ and $\Gamma(x)\Gamma(y)$ vanish simultaneously and, whenever they are nonzero, they differ by multiplication by an element of $\mathbb{K}^*$. This motivates the following definition.

\begin{defi}
    Let $\mathcal{S}$ be a semigroupoid and $M$ a $\mathbb{K}$-cancellative monoid. A \emph{projective representation} of $\mathcal{S}$ on $M$ is a map $\Gamma : \mathcal{S} \to M$ such that $\xi\Gamma : \mathcal{S} \to \mathcal{P}(M)$ is a semigroupoid homomorphism.
\end{defi}

Equivalently, we have the following characterization.

\begin{defi} \label{defrepprojcat}
    Let $\mathcal{S}$ be a semigroupoid and $M$ a $\mathbb{K}$-cancellative monoid. A projective representation of $\mathcal{S}$ on $M$ is a map $\Gamma : \mathcal{S} \to M$ such that
    \begin{enumerate}
        \item[(i)] If $(x,y) \in \mathcal{S}^{(2)}$, then
    \[
    \Gamma(xy) = 0 \iff \Gamma(x)\Gamma(y) = 0,
    \]
    and if $(x,y) \notin \mathcal{S}^{(2)}$, then $\Gamma(x)\Gamma(y) = 0$.
        
        \item[(ii)] There exists a partially defined map $\rho \colon \mathcal{S} \times \mathcal{S} \to \mathbb{K}^*$ with $\text{dom} (\rho) = \{ (x,y) \in \mathcal{S}^{(2)} : \Gamma(xy) \neq 0 \}$ such that
    \begin{equation} \label{iggammarho}
        \Gamma(x)\Gamma(y) = \Gamma(xy)\,\rho(x,y),
    \end{equation}
    for all $(x,y) \in \text{dom} (\rho)$.
\end{enumerate}
The map $\rho$ is called a \emph{factor set}.
\end{defi}

Let $\mathcal{S}$ be an inverse semigroupoid. We now introduce the notion of a \emph{partial} projective representation.

\begin{defi}
    Let $\mathcal{S}$ be an inverse semigroupoid and $M$ a $\mathbb{K}$-cancellative monoid. A \emph{partial projective representation} of $\mathcal{S}$ on $M$ is a map $\Gamma : \mathcal{S} \to M$ such that $\xi\Gamma : \mathcal{S} \to \mathcal{P}(M)$ is a partial semigroupoid homomorphism.
\end{defi}

Concerning Definition~\ref{defrepprojcat}, observe that if $x \in M$ and $a \in \mathbb{K}$, then
\[
xa = x(a1_M) = a(x1_M) = ax.
\]
Moreover, if $\Gamma(xy) = 0$ or $(x,y) \notin \mathcal{S}^{(2)}$, we may set $\rho(x,y) = 0_{\mathbb{K}}$ without affecting the validity of \eqref{iggammarho}. In this way, we may assume that dom$(\rho) = \mathcal{S} \times \mathcal{S}$. Furthermore, if $M$ is $\mathbb{K}$-cancellative, then $\rho$ is uniquely determined by $\Gamma$.

Similarly to the case of semigroups \cite{dokuchaev2010partialprojectiverep}, applying the definition of projective representation to the associativity relation
\[
\Gamma(x)\big(\Gamma(y)\Gamma(z)\big) = \big(\Gamma(x)\Gamma(y)\big)\Gamma(z),
\]
we obtain
\[
\rho(x,y)\rho(xy,z) = \rho(x,yz)\rho(y,z),
\]
for all $(x,y,z) \in \mathcal{S}^{(3)}$. We refer to this identity as the \emph{2-cocycle condition}.

\begin{theorem}
Let $\mathcal{C}$ be a category. A map $\rho : \mathcal{C} \times \mathcal{C} \to \mathbb{K}$ is a factor set for some projective representation of $\mathcal{C}$ if and only if
\begin{align} \label{cat3}
    \rho(x,y)\rho(xy,z) = \rho(x,yz)\rho(y,z)
\end{align}
and
\begin{align} \label{cat4}
    \rho(x,y) = 0 \iff \rho(r(x),xy) = 0,
\end{align}
for all $(x,y,z) \in \mathcal{C}^{(3)}$.
\end{theorem}
\begin{proof}
($\Rightarrow$): We already know that (\ref{cat3}) holds. Now,
\begin{align*}
    \rho(x,y) = 0 \iff \Gamma(xy) = 0 \iff \Gamma(r(x)xy) = 0 \iff \rho(r(x),xy) = 0.
\end{align*}

($\Leftarrow$): Consider the monoid with 0 $M = \mathcal{M}_{|\mathcal{C}|}(\mathbb{K})$, where $|\mathcal{C}|$ denotes the cardinality of $\mathcal{C}$. For all $x \in \mathcal{C}$, consider the $|\mathcal{C}| \times |\mathcal{C}|$ matrix  
\begin{align*}
    \Gamma(x) = (\gamma_{u,v}(x))_{u,v \in \mathcal{C}}
\end{align*}
defined by
\begin{align*}
    \gamma_{u,v}(x) = \begin{cases}
        \rho(u,x), \text{ if } (u,x) \in \mathcal{C}^{(2)} \text{ and } ux = v. \\
        0, \text{ otherwise.}
    \end{cases}
\end{align*}

The product between these matrices is always defined since they are monomial in the rows.

Take $x,y \in \mathcal{C}$. Write $\Gamma(x)\Gamma(y) = (\delta_{u,v})_{u,v \in \mathcal{C}}$. Then
\begin{align*}
    \delta_{u,v} = \sum_{t \in \mathcal{C}} \gamma_{u,t}(x)\gamma_{t,v}(y)
    & = \begin{cases}
    \rho(u,x)\gamma_{ux,v}(y), \text{ if } (u,x) \in \mathcal{C}^{(2)}, \\
    0, \text{ otherwise}
    \end{cases} \\
    & = \begin{cases}
    \rho(u,x)\rho(ux,y), \text{ if } (u,x,y) \in \mathcal{C}^{(3)} \text{  and } uxy = v, \\
    0, \text{ otherwise}.
    \end{cases}
\end{align*}

If $(x,y) \notin \mathcal{C}^{(2)}$ then $(u,x,y) \notin \mathcal{C}^{(3)}$, from where it follows that $\Gamma(x)\Gamma(y) = 0$ as we wanted. If $(u,x,y) \in \mathcal{C}^{(3)}$, we have that (\ref{cat3}) holds, so
\begin{align*}
    \delta_{u,v} & =  \begin{cases}
        \rho(u,xy)\rho(x,y), \text{ if } (u,x) \in \mathcal{C}^{(2)} \text{  and } uxy = v, \\
        0, \text{ otherwise}
    \end{cases} \\
    & = \gamma_{u,v}(xy)\rho(x,y),
\end{align*}
that is, $\Gamma(x)\Gamma(y) = \Gamma(xy)\rho(x,y)$.

In particular, $\Gamma(xy) = 0 \implies \Gamma(x)\Gamma(y) = 0$. By the construction of $\Gamma$, if $(x,y) \in \mathcal{C}^{(2)}$, we have that
\begin{align*}
    \Gamma(xy) = 0 \iff \rho(u,xy) = 0, \text{ for all } u \in \mathcal{C} \text{ such that } (u,x) \in \mathcal{C}^{(2)},
\end{align*}
and
\begin{align*}
    \Gamma(x)\Gamma(y) = 0 \iff \rho(u,xy)\rho(x,y) = 0, \text{ for all } u \in \mathcal{C} \text{ such that } (u,x) \in \mathcal{C}^{(2)}.
\end{align*}

It is enough to prove that $\rho(x,y) = 0$ implies that $\rho(u,xy) = 0$ for all suitable $u \in \mathcal{C}$. Assume that $\rho(x,y) = 0$ but $\rho(u,xy) \neq 0$ for some $u \in \mathcal{C}$ such that $(u,x) \in \mathcal{C}^{(2)}$ (if there is no such $u$, the result is proved). We have by (\ref{cat4}) that $\rho(r(u),uxy) \neq 0$, that is, $\rho(ux,y) \neq 0$. By (\ref{cat3}) it follows that $\rho(u,x) = \rho(r(u),ux) = 0$ and
\begin{align*}
    0 = \rho(r(u),ux)\rho(r(u)ux,y) = \rho(r(u),uxy)\rho(ux,y),
\end{align*}
that is, $\rho(r(u),uxy) = 0$ but $\rho(u,xy) \neq 0$, which is a contradiction with (\ref{cat4}).
\end{proof}

Consider two factor sets $\rho$ and $\sigma$. We say that  $\rho \sim \sigma$ if there is a map $\nu : \mathcal{C} \to \mathbb{K}^*$ such that
\begin{align*}
    \rho(x,y) = \nu(x)\nu(xy)^{-1}\nu(y)\sigma(x,y)
\end{align*}
whenever $(x,y) \in \mathcal{C}^{(2)}$. If $(x,y) \notin \mathcal{C}^{(2)}$, then $\rho(x,y) = \sigma(x,y) = 0$.

Defining the product of factor sets by pointwise multiplication, we have that $m(\mathcal{C})$, the set of all factor sets of $\mathcal{C}$, is a semigroup and $\sim$ is a congruence. The Schur multiplier $M(\mathcal{C})$ is then defined as the quotient semigroup $M(\mathcal{C}) = m(\mathcal{C})/\sim$.

By the same arguments used in \cite[Lemma 2]{dokuchaev2010partialprojectiverep}, it follows that the semigroups $m(\mathcal{C})$ and $M(\mathcal{C})$ are regular (every element has an inverse, not necessarily unique) and commutative. Hence, by Clifford's Theorem \cite[Corollary IV.2.2]{howie1976introduction}, they are strong semilattices of groups, in the sense that
\begin{align*}
    m(\mathcal{C}) = \bigcup_{\zeta \in b(\mathcal{C})} m_{\zeta}(\mathcal{C}), \quad M(\mathcal{C}) = \bigcup_{\zeta \in B(\mathcal{C})} M_{\zeta}(\mathcal{C}),
\end{align*}
where $b(\mathcal{C})$ and $B(\mathcal{C})$ are semilattices and $m_{\zeta}(\mathcal{C}), M_{\zeta}(\mathcal{C})$ are commutative groups.

This is the motivation to study the idempotents in $m(\mathcal{C})$.

\begin{prop}
    There is a one-to-one correspondence between the idempotents of $m(\mathcal{C})$ and the ideals of $\mathcal{C}$, given by $\varepsilon \longleftrightarrow \mathcal{I}_{\varepsilon}$, where
    \begin{align*}
        \varepsilon(x,y) = \begin{cases}
            0, \text{ if } (x,y) \in \mathcal{C}^{(2)} \text{ and } xy \in \mathcal{I}_\varepsilon, \\
            1, \text{ otherwise.}
        \end{cases}
    \end{align*}
    Moreover, for any idempotents $\varepsilon_1,\varepsilon_2 \in m(\mathcal{C})$, one has $\mathcal{I}_{\varepsilon_1\varepsilon_2} = \mathcal{I}_{\varepsilon_1} \cup \mathcal{I}_{\varepsilon_2}$.
\end{prop}
\begin{proof}
If $\varepsilon \in m(\mathcal{C})$ is idempotent, then $\varepsilon$ takes only the values $0$ and $1$. Hence, if $(x,y) \in \mathcal{C}^{(2)}$, we have $\varepsilon(x,y) = \varepsilon(r(x),xy)$. Moreover, if $(x,y,z) \in \mathcal{C}^{(3)}$, then
\begin{align*}
    \varepsilon(r(x),xy)\varepsilon(r(x),xyz) \overset{(*)}{=} \varepsilon(r(x),xyz)\varepsilon(r(y),yz).
\end{align*}

Define $\mathcal{I}_{\varepsilon} = \{ y \in \mathcal{C} : \varepsilon(r(y),y) = 0 \}$. Setting $x = r(y)$ in $(*)$, we obtain
\begin{align*}
    \varepsilon(r(y),y)\varepsilon(r(y),yz) = \varepsilon(r(y),yz)\varepsilon(r(y),yz) = \varepsilon(r(y),yz).
\end{align*} Therefore, if $y \in \mathcal{I}_{\varepsilon}$ and $yz \in \mathcal{C}^{(2)}$, then $\varepsilon(r(y),yz) = \varepsilon(r(yz),yz) = 0$, that is, $yz \in \mathcal{I}_{\varepsilon}$. 

On the other hand, setting $z = d(y)$ in $(*)$, we obtain $$\varepsilon(r(x),xy) = \varepsilon(r(x),xy)\varepsilon(r(y),y).$$ Hence, if $y \in \mathcal{I}_{\varepsilon}$ and $(x,y) \in \mathcal{C}^{(2)}$, then $xy \in \mathcal{I}_{\varepsilon}$, proving that $\mathcal{I}_{\varepsilon}$ is an ideal of $\mathcal{C}$.

Conversely, let $\mathcal{I}$ be an ideal of $\mathcal{C}$. Define
\begin{align*}
    \varepsilon(r(x),x) & = \begin{cases}
        1, \text{ if } x \notin \mathcal{I}, \\
        0, \text{ if } x \in \mathcal{I},
    \end{cases} \\
    \varepsilon(x,y) & = \begin{cases}
        \varepsilon(r(x),xy), \text{ if } xy \in \mathcal{C}^{(2)}, \\
        0, \text{ otherwise}.
    \end{cases}
\end{align*}

If $(x,y,z) \in \mathcal{C}^{(3)}$ is such that $xyz \in \mathcal{I}$, then $(*)$ holds. If $xyz \notin \mathcal{I}$, then neither $xy$ nor $yz$ belongs to $\mathcal{I}$, and in this case $(*)$ also holds.
\end{proof}

Denote by $Y(\mathcal{C})$ the semilattice of ideals of $\mathcal{C}$ with respect to inclusion, and note that we regard $\emptyset$ as an element of $Y(\mathcal{C})$. The next result is a direct consequence of the lemma above together with the fact that the relation $\sim$ separates idempotents.

\begin{corollary}
$b(\mathcal{C}) \cong B(\mathcal{C}) \cong Y(\mathcal{C})$ as semilattices.
\end{corollary}

It follows that the components of $m(\mathcal{C})$ and $M(\mathcal{C})$ can be indexed by the elements of $Y(\mathcal{C})$. Let $\varepsilon_{\mathcal{I}}$ denote the identity element of $m_{\mathcal{I}}(\mathcal{C})$. Then
\begin{align*}
    \varepsilon_{\mathcal{I}}(x,y) = 0 \iff (x,y) \notin \mathcal{C}^{(2)} \text{ or } (x,y) \in \mathcal{C}^{(2)} \text{ and } xy \in \mathcal{I}.
\end{align*}

Since $\rho \rho^{-1} = \varepsilon_{\mathcal{I}}$ for some $\mathcal{I} \in Y(\mathcal{C})$, we obtain the following result.

\begin{corollary}
The group $m_{\mathcal{I}}(\mathcal{C})$ consists precisely of the factor sets $\rho$ such that 
\begin{align*}
    \rho(x,y) = 0 \iff (x,y) \notin \mathcal{C}^{(2)} \text{ or } (x,y) \in \mathcal{C}^{(2)} \text{ and } xy \in \mathcal{I}.
\end{align*}
\end{corollary}

From now on, $\G$ will denote a groupoid. The next three results follow straightforwardly from the same arguments used in \cite[Section~4]{dokuchaev2010partialprojectiverep} and \cite{lata2020inverse}, and their proofs are therefore omitted.

\begin{prop}
A map $\Gamma :\G \to M$ is a partial projective representation if and only if $\Gamma = \overline{\Gamma} f$ for some projective representation $\overline{\Gamma} : \mathcal{E}(\G) \to M$.
\end{prop}

\begin{theorem} \label{teocondreppargrp}
Let $M$ be a $\mathbb{K}$-cancellative monoid. A map $\Gamma :\G \to M$ is a partial projective representation of $\G$ if and only if
\begin{enumerate}
    \item[(i)] For all $(g,h) \in \G \times\G$, if $(g,h) \in \G^{(2)}$,
    \begin{align*}
        \Gamma(g^{-1})\Gamma(gh) = 0 \iff \Gamma(g)\Gamma(h) = 0 \iff \Gamma(gh)\Gamma(h^{-1}) = 0,
    \end{align*}
    and if $(g,h) \notin \G^{(2)}$, $\Gamma(g)\Gamma(h) = 0$.
    
    \item[(ii)] There is a unique partially defined map $\sigma :\G \times\G \to \mathbb{K}^*$ such that
    \begin{align*}
        \text{\emph{dom}} \sigma = \{ (g,h) \in \G^{(2)} : \Gamma(g)\Gamma(h) \neq 0 \}
    \end{align*}
    and, for all $(g,h) \in \text{\emph{dom}} \sigma$,
    \begin{align*}
        \Gamma(g^{-1})\Gamma(g)\Gamma(h) & = \Gamma(g^{-1})\Gamma(gh)\sigma(g,h), \\
        \Gamma(g)\Gamma(h)\Gamma(h^{-1}) & = \Gamma(gh)\Gamma(h^{-1})\sigma(g,h).
    \end{align*}
\end{enumerate}

Besides that, for all $(g,h,k) \in \G^{(3)}$,
\begin{align*}
    \Gamma(g)\Gamma(h)\Gamma(k) \neq 0 \implies \sigma(g,h)\sigma(gh,k) = \sigma(g,hk)\sigma(h,k).
\end{align*}
\end{theorem}

\begin{prop}
A partially defined map $\sigma :\G \times\G \to \mathbb{K}^*$ is a factor set for some partial projective representation of $\G$ if and only if there is a factor set $\rho$ of $\mathcal{E}(\G)$ such that
\begin{enumerate}
    \item[(i)] For all $(g,h) \in \G \times\G$, $(g,h) \in \text{\emph{dom}} \sigma \iff ([g],[h]) \in \text{\emph{dom}} \rho$;
    
    \item[(ii)] For all $(g,h) \in \text{\emph{dom}} \sigma$, 
    \begin{align*}
        \sigma(g,h) = \frac{\rho([g],[h])\rho([g^{-1}],[g][h])}{\rho([g^{-1}],[gh])}.
    \end{align*}
\end{enumerate}
\end{prop}

The proposition above shows that the factor sets of a groupoid $\G$ form a regular, commutative semigroup under pointwise multiplication, which we denote by $pm(\G)$. One may also define the partial Schur multiplier by $pM(\G) = pm(\G)/\sim$. By Clifford's Theorem, these semigroups are isomorphic to strong semilattices of abelian groups. For this reason, it is important to study the idempotents of $pm(\G)$.

Let $\sigma$ be a factor set of $\G$ and let $D = \text{dom} (\sigma)$. The next result describes basic properties of the domain of factor sets. Its proof is omitted, since it follows the same arguments (with minor adaptations to the groupoid setting) as in \cite[Proposition~4]{dokuchaev2010partialprojectiverep}.

\begin{prop}\label{prop414} If $(g,h) \in \G \times \G$,
\begin{align*}
    (g,h) \in D \iff (g^{-1},gh) \in D & \iff (gh,h^{-1}) \in D \iff (h,h^{-1}g^{-1}) \in D \\
    \iff (h^{-1},g^{-1}) \in D & \iff (h^{-1}g^{-1},g) \in D,
\end{align*}
and
\begin{align*}
    (g,d(g)) \in D \iff (g^{-1},g) \in D \iff (d(g), g^{-1}) \in D \iff (r(g),g) \in D.
\end{align*}    
\end{prop}

From now on, we assume that $\Gamma(r(g))$ and $\Gamma(d(g))$ act as left and right identities for $\Gamma(g)$, respectively (not necessarily equal to $1_M$). In particular, this implies that $\sigma(e,e) = 1$ for all $e \in \G_0$.

\begin{prop}
If $(g,h) \in D$, then $(g,d(g)),(r(h),h) \in D$ and $$\sigma(g,d(g)) = \sigma(r(g),g) = \sigma(h,d(h)) = \sigma(r(h),h) = 1.$$
\end{prop}
\begin{proof}
    We have that
    \begin{align*}
        \Gamma(g) = \Gamma(g)d(\Gamma(g)) & = \Gamma(g)\Gamma(d(g)) = \Gamma(g)\Gamma(d(g))\Gamma(d(g)) \\
        & = \Gamma(gd(g))\Gamma(d(g))\sigma(g,d(g)) = \Gamma(g)\sigma(g,d(g)).
    \end{align*}
    
    Since $\Gamma(g) \neq 0$, it follows that $\sigma(g,d(g)) = 1$. By Proposition \ref{prop414}, $(r(g),g) \in D$, hence $\sigma(r(g),g) = 1$. Similarly $(h,d(h)), (r(h),h) \in D$ in $\sigma(h,d(h)) = \sigma(r(h),h) = 1$.
\end{proof}

%The next theorem will characterize the idempotents of $pm(\G)$, but first we need a technical lemma.

\begin{lemma}\label{lema416}
Assume that the values of $\sigma :\G \times\G \to \mathbb{K}$ are 0 and 1, $\sigma(e,e) = 1$ for all $e \in \G_0$ and for all $(g,h) \in \text{\emph{dom}} (\sigma)$,
\begin{align} \label{impl12}
    (gh,h^{-1}),(h^{-1},g^{-1}),(g,d(g)) \in \text{\emph{dom}} (\sigma).
\end{align}

Let $(a,b),(c,d) \in \G^{(2)}$ be such that $\sigma(a,b) = 0$ and $\{ a^{-1}, b \} \subseteq \{ d(c), c^{-1}, d \}$. Then $\sigma(c,d) = 0$.
\end{lemma}
\begin{proof}
    First, \eqref{impl12} yields that
    \begin{align}
        (g,h) \in D \iff & (gh,h^{-1}) \in D  \iff (g^{-1},gh) \in D \label{impl13} \\ 
        (g,h) \in D & \implies (r(h),h) \in D. \label{impl14}
    \end{align}
    
    Assume that $\sigma(c,d) = 1$. By (\ref{impl12}), (\ref{impl13}) and (\ref{impl14}) we obtain
    \begin{align*}
        \sigma(d^{-1},c^{-1}) = \sigma(c,d(c)) = \sigma(d(c),d) = \sigma(d(c),c^{-1}) = \sigma(d^{-1},d(c)) = 1.
    \end{align*}
    
     Moreover, $\sigma(d(c),d(c)) = 1$ by assumption. Hence we have that $\sigma(a,b)$ may coincide only with $\sigma(c,c^{-1})$ or $\sigma(d^{-1},d)$. 
    
    Applying (\ref{impl12}) and (\ref{impl14}) in $(c^{-1},cd) \in \text{dom} (\sigma)$, we obtain $(c,c^{-1}) \in \text{dom} (\sigma)$. Similarly, starting with $(d^{-1},c^{-1}) \in \text{dom} (\sigma)$ we obtain $(d^{-1},d) \in \text{dom} (\sigma)$.
    
    Therefore $\sigma(c,c^{-1}) = \sigma(d^{-1},d) = 1$, a contradiction.
\end{proof}

\begin{theorem}
    A map $\sigma :\G \times\G \to \mathbb{K}$ such that $\sigma(e,e) = 1$, for all $e \in \G_0$, is an idempotent factor set if and only if its values are 0 and 1 and \eqref{impl12} holds.
\end{theorem}
\begin{proof}
    We already proved $(\Rightarrow)$. For the converse, write $\mathcal{L} = \{ [g][h] : \sigma(g,h) = 0 \}$, $$\mathcal{I} = \mathcal{E}(\G) \mathcal{L} \mathcal{E}(\G) = \{ \alpha \ell \beta : \alpha,\beta \in \mathcal{E}(\G), \ell \in \mathcal{L} \text{ and } \exists \alpha\ell\beta \}$$ and define $\rho : \mathcal{E}(\G)^2 \to \mathbb{K}$ by
    \begin{align*}
        \rho(\alpha,\beta) = \begin{cases}
            1, \text{ if } \alpha\beta \notin \mathcal{I}, \\
            0, \text{ if } \alpha\beta \in \mathcal{I}.
        \end{cases}
    \end{align*}
    
    Since $\mathcal{I}$ is an ideal, it follows that $\rho$ is an idempotent factor set of $\mathcal{E}(\G)$. It only remains to us to prove that $\rho([u],[v]) = \sigma(u,v)$, for all $(u,v) \in \G^{(2)}$.
    
    If $\sigma(u,v) = 0$, then $[u][v] \in \mathcal{L}$ and $[u][v] = [r(u)][u][v][d(v)] \in \mathcal{I}$, from where it follows that $\rho([u],[v]) = 0$. Assume that $\rho([u],[v]) = 0$. Then $[u][v] \in \mathcal{I}$. Moreover, there are $\alpha, \beta \in \mathcal{E}(\G)$, $x,y \in \G$ such that $\exists \alpha[x][y]\beta$, $\sigma(x,y) = 0$ and $[u][v] = \alpha[x][y]\beta$.
    
    Let $\alpha = \varepsilon_{g_1} \cdots \varepsilon_{g_n}[g]$ and $\beta = \varepsilon_{h_1} \cdots \varepsilon_{h_m}[h]$, with $g_i, h_i, g, h \in \G$, $r(g_i) = r(g)$ and $r(h_i) = r(h)$ in the standard form \cite[Proposition 2.7]{lata2020inverse}. Since $\exists \alpha[x][y]\beta$, we also have $d(g) = r(x)$, $d(x) = r(y)$, $d(y) = r(h)$. We will use the uniqueness of the decomposition in the standard form in $\mathcal{E}(\G)$ to conclude the demonstration.
    
    On the one hand, $[u][v] = [u][u^{-1}][u][v] = [u][u^{-1}][uv] = \varepsilon_u[uv]$. On the other hand,
    \begin{align*}
        \alpha[x][y]\beta & = \varepsilon_{g_1} \cdots \varepsilon_{g_n}[g][x][y]\varepsilon_{h_1} \cdots \varepsilon_{h_m}[h] \\
        & = \varepsilon_{g_1} \cdots \varepsilon_{g_n} \varepsilon_{gxyh_1} \cdots \varepsilon_{gxyh_m}[g][x][y][h] \\
        & = \varepsilon_{g_1} \cdots \varepsilon_{g_n} \varepsilon_{gxyh_1} \cdots \varepsilon_{gxyh_m}\varepsilon_{g}\varepsilon_{gx}\varepsilon_{gxy}[gxyh].
    \end{align*}
    
    By the uniqueness of the decomposition in $\mathcal{E}(\G)$, we obtain 
    \begin{align*}
        \{ g, gx, gxy \} \subseteq \{ r(u), u, uv \}.
    \end{align*} Hence $g = r(u), u$ or $uv$. We will handle these cases separately.
    
    If $g = r(u)$, then $r(u) = d(g) = r(x)$, from where it follows that $\{ x, xy \} \subseteq \{ r(u), u, uv\}$. In the Lemma \ref{lema416}, consider $a = x^{-1}$, $b = xy$, $c = u^{-1}$ and $d = uv$. Thus $\sigma(x^{-1},xy) = 0 \Rightarrow \sigma(u,v) = \sigma(u^{-1},uv) = \sigma(c,d) = 0$.
    
    If $g = u$, then $\{ u, ux, uxy \} \subseteq \{ r(u), u, uv \}$. Multiplying by $u^{-1}$ on the left we obtain $\{ x, xy \} \subseteq \{ d(u), u^{-1}, v\}$. Considering $a = x^{-1}$, $b = xy$, $c = u$, $d = v$ as in the Lemma \ref{lema416}, we obtain $\sigma(u,v) = 0$.
    
    Similarly, if $g = uv$, we have $ \{ x, xy \} \subseteq \{ d(v), v^{-1}, (uv)^{-1} \}$. Taking $a = x^{-1}$, $b = xy$, $c = v$, $d = (uv)^{-1}$ we obtain $\sigma(u,v) = \sigma(v^{-1},u^{-1}) = \sigma(v,v^{-1}u^{-1}) = \sigma(c,d) = 0$.
\end{proof}

Let $\Gamma :\G \to M$ be a partial projective representation of $\G$ on a $\mathbb{K}$-cancellative monoid $M$ and $\sigma :\G \times\G \to \mathbb{K}$ the factor set associated with $\Gamma$. Assume that $\Gamma(r(g))$ and $\Gamma(d(g))$ are left and right identities to $\Gamma(g)$, respectively.

From $(\Gamma(g^{-1})\Gamma(g))\Gamma(g^{-1}) = \Gamma(g^{-1})(\Gamma(g))\Gamma(g^{-1}))$ we obtain $\sigma(g,g^{-1}) = \sigma(g^{-1},g),$ for all $g \in \G$. Write $\eta_g = \Gamma(g)\Gamma(g^{-1})$. Define 
\begin{align*}
    n_g = \begin{cases}
        \eta_g\sigma(g^{-1},g)^{-1}, \text{ if } \Gamma(g) \neq 0, \\
        0, \text{ if } \Gamma(x) = 0.
    \end{cases}
\end{align*}

The properties of $\eta_g$ and $n_g$ will be stated in the next lemma and the computations to demonstrate them are analogous to the case of groups \cite[Lemma 7]{dokuchaev2010partialprojectiverep}.

\begin{lemma}
Let $g,h \in \G$. We have that
\begin{align*}
    \eta_g^2 = \eta_g\sigma(g,g^{-1}), & \text{ } n_g^2 = n_g,  \\
    \eta_g = 0 \iff \sigma(g,g^{-1}) = 0 \iff & \Gamma(g) = 0 \iff n_g = 0, \\
    r(g) = r(h) \implies (\eta_g\eta_h = 0  \iff & \eta_h\eta_g = 0 \iff n_gn_h = 0), \\
    d(g) = r(h) \implies \Gamma(g)n_h & = n_{gh}\Gamma(g), \\
    r(g) = r(h) \implies n_gn_h & = n_hn_g.
\end{align*}
\end{lemma}

\section{Twisted Partial Actions on $\mathbb{K}$-monoids}
In Definition \ref{def21}, we introduced the notion of twisted partial actions of groupoids on rings. We now extend this framework by defining twisted partial actions of groupoids on semigroups, and, in particular, on $\mathbb{K}$-monoids.

\begin{defi} \label{deftpak}
A \emph{partial groupoid action $\theta= ( S_g, \theta_g)_{g \in \G}$ of $\G$ on a semigroup $S$} is a collection of ideals $S_x$ and semigroup isomorphisms $\theta_g : S_{g^{-1}} \to S_g$ such that, given $(g,h) \in \G^{(2)}$, 
\begin{enumerate}
   \item[(i)] $S_g \subseteq S_{r(g)}$ and $\theta_e = \text{Id}_{S_e}$, for all $e \in \G_0$;

    \item[(ii)] $\theta_g(S_{g^{-1}} \cap S_h) = S_g \cap S_{gh}$;
    
    \item[(iii)] $\theta_g \circ \theta_h(s) = \theta_{gh}(s)$, for all $s \in S_{h^{-1}} \cap S_{(gh)^{-1}}$.\end{enumerate}
\end{defi}

Although one can define twisted partial actions of groupoids on semigroups (and on $\mathbb{K}$-semigroups), in our setting we are primarily interested in monoids, and in particular in $\mathbb{K}$-monoids. Accordingly, we will not pursue the more general framework here.

\begin{defi}\label{def_pa_monoid}
Let $S$ be a $\mathbb{K}$-monoid and $\theta$ a partial action of $\G$ on $S$ such that every ideal $S_g$ is a monoid and every $\theta_g$ is a $\mathbb{K}$-map. A $\mathbb{K}$-\emph{valued twisting related to} $\theta$ is a map $\sigma :\G \times \G \to \mathbb{K}$ such that
\begin{enumerate}
    \item[(iv)] $\sigma(g,h) = 0 \iff S_g \cap S_{gh} = 0$, for all $(g,h) \in \G^{(2)}$;
    
    \item[(v)] $\sigma(g,d(g)) = \sigma(r(g),g) = 1$;
    
    \item[(vi)] $S_{g} \cap S_{gh} \cap S_{ghk} \neq 0 \implies \sigma(g,h)\sigma(gh,k) = \sigma(h,k)\sigma(g,hk)$ with $(g,h,k) \in \G^{(3)}$.
\end{enumerate}

If the pair $(\theta, \sigma)$ satisfies $(i) - (iii)$ of Definition \ref{deftpak} and, in addition, satisfies $(iv) - (vi)$, then it is called a \emph{groupoid twisted partial action of $\G$ on $S$}.
\end{defi}

\begin{obs}
    We have that $S_g \cap S_h = S_gS_h$, for all $(g,h) \in \G^{(2)}$. In fact, every $S_g$ is generated by a central idempotent, namely $1_g$. Notice that if $S_gS_{gh}S_{ghk} \neq 0$, then all the values of $\sigma$ in Definition \ref{deftpak}(vi) are non-zero.
\end{obs}

Let $A = (A,+,\cdot)$ be a $\mathbb{K}$-cancellative $\mathbb{K}$-algebra, that is, an $\mathbb{K}$-algebra such that the $\mathbb{K}$-monoid $A' = (A,\cdot)$ is $\mathbb{K}$-cancellative. For each unital ideal $I$ of $A$  with identity $1_I \neq 0$, we assume $\mathbb{K} \subseteq I$ under the isomorphism $\mathbb{K} \to \mathbb{K}1_I$.

\begin{prop} \label{propcompalgmon} Under the hypothesis above, the following statements are valid.
    \begin{enumerate}
        \item[(a)] Let $\alpha = (\{D_g\}_{g \in \mathcal{G}}, \{\alpha_g\}_{g \in \mathcal{G}}, \{w_{g,h}\}_{(g,h) \in \mathcal{G}^{(2)}})$ be a groupoid twisted partial action of $\G$ on $A$. If every $w_{g,h} := \kappa(g,h) \in \mathbb{K}$, then defining $\sigma(g,h) = \kappa(g,h)$ for all $(g,h) \in \G^{(2)}$ we have that $\alpha' = (\alpha, \sigma)$ is a groupoid twisted partial action of $\G$ on $A'$.
        
        \item[(b)] Let $\theta' = (A_g, \theta_g, \sigma)_{g \in \G}$ be a groupoid twisted partial action of $\G$ on $A'$. If $(A_g, +, \cdot)$ is an ideal of $A$, $A = \sum_{e \in \G_0} A_e$, and every $\theta_g$ is a ring homomorphism, then $\theta = (\{A_g\}_{g \in \mathcal{G}}, \{\theta_g\}_{g \in \mathcal{G}}, \{w_{g,h}\}_{(g,h) \in \mathcal{G}^{(2)}})$ is a groupoid twisted partial action of $\G$ on $A$, where $w_{g,h} = \sigma(g,h) \in \mathbb{K}$.
    \end{enumerate}
    
    In particular, if every element of $A_g = D_g$ (set equality) is of the form $k1_g$, where $1_g$ is the identity of $A_g$ and $k \in \mathbb{K}$, then there is a one-to-one correspondence between the groupoid twisted partial actions of $\G$ on $A$ and the groupoid twisted partial actions of $\G$ on $A'$.
\end{prop}
\begin{proof}
(a): It is clear that if $D_g$ is an unital ideal of $A$ then $D_g$ is an ideal of $A'$ that is a monoid. Furthermore, the ring isomorphism $\alpha_g : D_{g^{-1}} \to D_g$ can be seen as a monoid isomorphism. So the conditions (i) and (ii) of Definition \ref{deftpak} hold. Notice that given $b \in \mathbb{K}$
\begin{align*}
    (a)R_b = ab = a(b1_A) = b(a1_A) = ba = L_b(a),
\end{align*}
for all $a \in A$, so $R_b \equiv L_b$. In this way, we have that 
\begin{align*}
    \alpha_g \circ \alpha_h(a) & = w_{g,h}\alpha_{gh}(a)w_{g,h}^{-1} = 
    L_{\kappa(g,h)}\alpha_{gh}(a)R_{\kappa(g,h)}^{-1}\\ 
    & = L_{\kappa(g,h)}\alpha_{gh}(a)R_{\kappa(g,h)^{-1}} = L_{\kappa(g,h)}L_{\kappa(g,h)^{-1}}\alpha_{gh}(a) = \alpha_{gh}(a),
\end{align*}
for all $(g,h) \in \G^{(2)}$, $a \in A$, so that (iii) of Definition \ref{deftpak} holds. 

Define, then, $\sigma(g,h) = \kappa(g,h)$, for all $(g,h) \in \G^{(2)}$. Now we will prove the conditions (iv)-(vi) of Definition \ref{def_pa_monoid}. 

(iv): If $D_gD_{gh} = 0$, we have that the only possible multiplier is $w_{g,h} = 0$. Hence, $\sigma(g,h) = \kappa(g,h) = 0_\mathbb{K}$. By the other hand, if $\sigma(g,h) = 0_\mathbb{K}$, then $w_{g,h} = 0$ is an invertible element in $D_gD_{gh}$, which implies $D_gD_{gh} = 0$.

(v): It follows directly from Definition \ref{def21} (v).

(vi): By Definition \ref{def21} (vi) we have that $\alpha_g(aw_{h,k})w_{g,hk} = \alpha_g(a)w_{g,h}w_{gh,k}$,
for all $(g,h,k) \in \G^{(3)}$ and $a \in D_{g^{-1}}D_hD_{hk}$. Applying the multipliers we obtain
\begin{align*}
     \sigma(g,hk)\alpha_g(\sigma(h,k)a) = \sigma(g,h)\sigma(gh,k)\alpha_g(a).
\end{align*}

Since $\alpha_g$ is a $\mathbb{K}$-algebra isomorphism, it is a $\mathbb{K}$-map. Hence,
\begin{align*}
    \sigma(h,k)\sigma(g,hk)\alpha_g(a) = \sigma(g,h)\sigma(gh,k)\alpha_g(a).
\end{align*}

However, $A$ is $\mathbb{K}$-cancellative. Taking $a = 1_{g^{-1}}1_h1_{hk}$, it implies that $\alpha_g(a) \neq 0$, then
\begin{align*}
    \sigma(h,k)\sigma(g,hk) = \sigma(g,h)\sigma(gh,k),
\end{align*}
for all $(g,h,k) \in \G^{(3)}$.

(b): By a similar argumentation of (a), taking $w_{g,h} = (R_{\sigma(g,h)}, L_{\sigma(g,h)})$, we have that the items (i)-(vi) of Definitions \ref{deftpak} and \ref{def_pa_monoid} imply (i)-(vi) of Definition \ref{def21} if and only if the ideals $A_g$ and isomorphisms $\theta_g$ are compatible with the sum in $A$.

Now, if every element of $A_g = D_g$ is of the form $\ell1_g$, $g \in \G$, $\ell \in \mathbb{K}$, then every element in $D_gD_{gh}$ is of the form $\ell1_g1_{gh}$, for $(g, h) \in \G^{(2)}$, $\ell \in \mathbb{K}$. In fact, if $a \in D_gD_{gh}$, thus $a = \sum b_ic_i$, where $b_i \in D_g$, $c_i \in D_{gh}$. By assumption, $b_i = \ell_{b_i}1_g$ and $c_i = \ell_{c_i}1_{gh}$ for some $\ell_{b_i}, \ell_{c_i} \in \mathbb{K}$. Thus,
\begin{align*}
    a = \sum (\ell_{b_i}\ell_{c_i})1_g1_{gh} = \left ( \sum \ell_{b_i}\ell_{c_i} \right )1_g1_{gh} = \ell1_g1_{gh},
\end{align*}
where $\ell = \sum \ell_{b_i}\ell_{c_i} \in \mathbb{K}$. This writing is unique since $A$ is $\mathbb{K}$-cancellative. Therefore every multiplier is of the form we need in (a). 

Furthermore, we can induce a sum in $A_g$ by the sum in $\mathbb{K}$. In fact, $\ell_1 1_g + \ell_2 1_g = (\ell_1 + \ell_2)1_g$, for all $\ell_1,\ell_2 \in \mathbb{K}$, $g \in \G$. Since $\theta_g$ a $\mathbb{K}$-map, we obtain
\begin{align*}
    \theta_g(\ell_1 1_{g^{-1}} + \ell_2 1_{g^{-1}}) & = \theta_g((\ell_1 + \ell_2)1_{g^{-1}}) = (\ell_1 + \ell_2)\theta_g(1_g^{-1}) \\
    & = (\ell_1 + \ell_2)1_g = \ell_11_g + \ell_21_g = \theta_g(\ell_11_{g^{-1}}) + \theta_g(\ell_21_{g^{-1}}),
\end{align*}
that is, $\theta_g$ is a ring isomorphism and $A_g$ is a ring ideal, giving us the necessary conditions to apply (b). Then, by items (a) and (b), the map $\alpha \mapsto \alpha'$ is clearly a bijection between the groupoid twisted partial actions of $\G$ on $A$ and the groupoid twisted partial actions of $\G$ on $A'$.
\end{proof}

At first glance, the definition of a twisted partial action of a groupoid on a $\mathbb{K}$-monoid may not appear to differ from that of a (non-twisted) partial action, since the twisting $\sigma$ does not intervene in the composition of the partial isomorphisms. However, the essential distinction becomes apparent in the construction of the associated crossed products.

Let $(\theta,\sigma)$ be a  groupoid twisted partial action of $\G$ on a $\mathbb{K}$-monoid $S$. Set $$L = \{ a \delta_g : a \in S_g, g \in \G \} \cup \{ 0 \},$$ where the $\delta_g$'s are symbols. Define a multiplication on $L$ by
\begin{align*}
    (a\delta_g)(b\delta_h) = \begin{cases}
        \theta_g(\theta_{g}^{-1}(a)b)\sigma(g,h) \delta_{gh}, \text{ if } (g,h) \in \G^{(2)}, \\
        0, \text{ otherwise}
    \end{cases}
\end{align*}
and extend it by setting  $0 \cdot a\delta_g = a\delta_g \cdot 0 = 0$, for all $g \in \G, a \in S_g$.

By Theorem \ref{teoassoc}, this operation is associative, so $L$ is a semigroup. We define the crossed product $S *_{\theta,\sigma}\G$ as the quotient semigroup $L/I$, where $I = \{ 0 \delta_g : g \in \G \} \cup \{ 0 \}$ is an ideal of $L$.

\begin{prop}
    Let $(A,+,\cdot)$ be a $\mathbb{K}$-cancellative $\mathbb{K}$-algebra and consider $A' = (A,\cdot)$ the $\mathbb{K}$-cancellative monoid associated with $A$. Let $\alpha = (\{D_g\}_{g \in \mathcal{G}}, \{\alpha_g\}_{g \in \mathcal{G}}, \{w_{g,h}\}_{(g,h) \in \mathcal{G}^{(2)}})$ be a groupoid twisted partial action of a groupoid $\G$ on $A$ such that $\G_0$ is finite and $w_{g,h} = \kappa(g,h) \in \mathbb{K}$ as in Proposition \ref{propcompalgmon}(a). Then there is a monomorphism $A' *_{\alpha',\sigma}\G \to (A *_{\alpha} \G)'$, where $\alpha'$ and $\sigma$ are defined as in Proposition \ref{propcompalgmon}(a) and $(A *_{\alpha} \G)'$ is the $\mathbb{K}$-cancellative monoid obtained from $A *_{\alpha} \G$.
\end{prop}
\begin{proof}
    By Theorem \ref{teoassoc}, we have that $A *_{\alpha} \G$ is an $\mathbb{K}$-algebra with unity $1_A = \sum_{e \in \G_0} 1_e\delta_e$.
    
    The fact that $A *_{\alpha} \G$ is $\mathbb{K}$-cancellative follows directly from the fact that $A$ is $\mathbb{K}$-cancellative. So we can consider the $\mathbb{K}$-cancellative monoid $(A *_{\alpha} \G)'$ by dropping the sum in $A *_{\alpha} \G$.
    
    Consider the map
    \begin{align*}
        \varphi : A' *_{\alpha',\sigma} \G & \to (A *_{\alpha} \G)' \\
        a\delta_g & \mapsto a\delta_g.
    \end{align*}
    
     We want to show that $\varphi$ is an monomorphism of $\mathbb{K}$-semigroups. Clearly, $\varphi$ is injective. Let $a\delta_x, b\delta_y \in A' \rtimes_{\alpha',\sigma} \G$. Then
    \begin{align*}
        \varphi((a\delta_x)(b\delta_y)) & = \begin{cases}
        \varphi(\alpha'((\alpha')_x^{-1}(a)b)\sigma(x,y)\delta_{xy}), \text{ if } \exists xy \\
        0, \text{ otherwise } 
        \end{cases} \\
        & = \begin{cases}
        \alpha'((\alpha')_x^{-1}(a)b)\sigma(x,y)\delta_{xy}, \text{ if } \exists xy \\
        0, \text{ otherwise } 
        \end{cases} \\
        & = \begin{cases}
        \alpha(\alpha_x^{-1}(a)b)w_{x,y}\delta_{xy}, \text{ if } \exists xy \\
        0, \text{ otherwise } 
        \end{cases} \\
        & = (a\delta_x)(b\delta_y) \\
        & = \varphi(a\delta_x)\varphi(b\delta_y).
    \end{align*}
\end{proof}

\begin{obs}
In the notation of the proposition above, the monomorphism $\varphi$ is an isomorphism if and only if $\G = \{e\}$ is the trivial group. Indeed, the implication $(\Leftarrow)$ is clear. For the converse, assume that $\G$ has at least two distinct elements $x \neq y$. Then $a\delta_x + b\delta_y \in A \rtimes_\alpha \G$, for some $a \in A_x = D_x$, $b \in A_y = D_y$, but $a\delta_x + b\delta_y \neq c\delta_z$ for any $z \in \G$, $c \in A_z = D_z$. Hence $a\delta_x + b\delta_y \notin A' \rtimes_{\alpha',\sigma} \G$, which shows that $\varphi$ cannot be surjective.
\end{obs}

\section{The relation between partial projective representations and twisted partial actions}

Let $\Gamma$ be a partial projective representation of $\G$ on a $\mathbb{K}$-cancellative monoid $M$. Denote by $\Gamma(\G)$ the subsemigroup of $M$ generated by the elements $a\Gamma(g)$, with $a \in \mathbb{K}$ and $g \in \G$, and let $S$ be the commutative subsemigroup of $\Gamma(\G)$ generated by the elements $\alpha n_g$, with $\alpha \in \mathbb{K}$ and $g \in \G$. Define $S_g = S n_g$. Then
\begin{align*}
    S_g = 0 
    \iff n_g = 0 
    \iff \Gamma(g) = 0 
    \iff \Gamma(g^{-1}) = 0 
    \iff S_{g^{-1}} = 0 .
\end{align*}

Moreover, it is easy to see that, for all $(g,h,k) \in \G^{(3)}$,
\begin{align*}
    S_g S_{gh} S_{ghk} = 0 
    \iff \Gamma(g)\Gamma(h)\Gamma(k) = 0 .
\end{align*} Furthermore, if $r(g) \neq r(h)$, then $n_g n_h = 0$, and hence $S_g S_h = 0$.

We can now show how to construct a twisted partial action from a partial projective representation.

\begin{prop} \label{proppprtotpa}
    Let $g \in \G$. Define $\theta_g : S_{g^{-1}} \to S_g$ by
    \begin{align*}
        \theta_g(s) = \begin{cases}
        \Gamma(g)s\Gamma(g^{-1})\sigma(g^{-1},g)^{-1}, & \text{ if } \Gamma(g^{-1}) \neq 0, \\
            0, & \text{ if } \Gamma(g^{-1}) = 0.
        \end{cases}
    \end{align*} Then $\theta = \theta^{\Gamma} = (S_g, \theta_g)_{g \in \G}$ is a partial action of $\G$ on $S$ and the factor set $\sigma$ is a twisting related to $\theta$. Moreover, the map
    \begin{align*}
        \psi : S *_{\theta,\sigma}\G & \to \Gamma(\G) \\
        a\delta_g & \mapsto a\Gamma(g)
    \end{align*}
    is an epimorphism of semigroups.
\end{prop}
\begin{proof}
    The proof that $(\theta,\sigma)$ defines a twisted partial action is analogous to the one given in \cite[Theorem 6]{dokuchaev2010partialprojectiverep}. Therefore, it suffices to verify that $S_g \subseteq S_{r(g)}$ for all $g \in \G$ and that $\theta_e = \text{Id}_{S_e}$. This follows from the fact that $n_{r(g)} = \Gamma(r(g))$ is a two-sided identity for $n_g$.
    
    Let us verify that $\psi$ is an epimorphism of semigroups. Let $a = a n_g \in S_g$ and $b = b n_h \in S_h$. We have
    \begin{align*}
        \psi(a \delta_g b \delta_h) & = \begin{cases}
            \theta_g(\theta_g^{-1}(a)b)\sigma(g,h)\Gamma(gh), \text{ if } (g,h) \in \G^{(2)}, \\
            0, \text{ otherwise}
        \end{cases} \\
        & = \begin{cases}
            \eta_ga\Gamma(g)b\Gamma(g^{-1})\Gamma(g)\Gamma(h)\sigma(g^{-1},g)^{-2}, \text{ if } (g,h) \in \G^{(2)}, \\
            0, \text{ otherwise}
        \end{cases}       
        \end{align*}

        \begin{align*}
        & = \begin{cases}
            \eta_ga\Gamma(g)b\eta_{g^{-1}}\Gamma(h)\sigma(g^{-1},g)^{-2}, \text{ if } (g,h) \in \G^{(2)}, \\
            0, \text{ otherwise}
        \end{cases} \\
        & = \begin{cases}
            n_ga\Gamma(g)bn_{g^{-1}}\Gamma(h), \text{ if } (g,h) \in \G^{(2)}, \\
            0, \text{ otherwise}
        \end{cases} \\
        & = \begin{cases}
            a\Gamma(g)b\Gamma(h), \text{ if } (g,h) \in \G^{(2)}, \\
            0, \text{ otherwise}
        \end{cases} \\
        & = \psi(a\delta_g)\psi(b\delta_h),
    \end{align*}
    since $\Gamma(g)n_{g^{-1}} = \Gamma(g)$ and $S$ is commutative. This shows that $\psi$ is a homomorphism.
    
    Consider $a = \ell \Gamma(g_1)\Gamma(g_2) \cdots \Gamma(g_m) \in \Gamma(\G)$, where $\ell \in \mathbb{K}$, $(g_1, \ldots , g_m) \in \G^{(m)}$. We have that
    \begin{align*}
        a = \ell' n_{g_1}n_{g_1g_2} \cdots n_{g_1g_2 \cdots g_m}\Gamma(g_1g_2 \cdots g_m) \in S_g\Gamma(g),
    \end{align*}
    where $g = g_1 \cdots g_m$ and $\ell' \in \mathbb{K}$. Hence, $\Gamma(\G) \subseteq \cup_{g \in \G} S_g\Gamma(g)$. If $(g_1, \ldots , g_m) \notin\G^{(m)}$, then $a = 0$. The reverse inclusion is immediate, and hence $\psi$ is surjective.
\end{proof}

Let $\theta$ be a partial action of $\G$ on a $\mathbb{K}$-monoid $T$ with twisting $\sigma$. Each ideal $T_g$ is generated by $1_g$, a central idempotent. It follows that
\begin{align*}
    T_{g_1} \cdots T_{g_m} = T1_{g_1} \cdots 1_{g_m},
\end{align*}
for all $(g_1, \ldots, g_m) \in \G^{(m)}$. Then,
\begin{align*}
    \theta_g(1_{g^{-1}}1_{h}1_k) = 1_g1_{gh}1_{gk},
\end{align*}
for all $(g,h,k) \in \G^{(3)}$.

\begin{prop} \label{proptpatoppr}
    The map 
    \begin{align*}
        \Gamma_\theta :\G & \to T *_{\theta, \sigma}\G \\
        g & \mapsto 1_g\delta_g,
    \end{align*}
    is a partial projective representation with factor set $\sigma$.
\end{prop}
\begin{proof}
    We already know that $\sigma(r(g),g) = \sigma(g,d(g)) = 1$, for all $g \in \G$. It follows that $\Gamma_{\theta}(r(g)) = 1_{T_{r(g)}}\delta_{r(g)} = (1_{T_{r(g)}}\delta_{r(g)})^2 = \Gamma_{\theta}(r(g))^2$ is a left identity for $\Gamma_\theta(g) = 1_g\delta_g$. Analogously, one has $\Gamma_\theta(g)\Gamma_\theta(d(g)) = \Gamma_\theta(g)$. If $(g,h) \notin \G^{(2)}$, then $\Gamma_\theta(g)\Gamma_\theta(h) = 1_g\delta_h 1_g\delta_h = 0$. For $(g,h) \in \G^{(2)}$, the claim follows by the same arguments as in \cite[Theorem 8]{dokuchaev2010partialprojectiverep}.
\end{proof}

The next result establishes the compatibility between the constructions given in Propositions \ref{proppprtotpa} and \ref{proptpatoppr}.

\begin{prop} \label{propcomp}
    Every partial projective representation $\Gamma :\G \to M$ of $\G$ on a $\mathbb{K}$-cancellative monoid $M$ can be recovered from $\Gamma_{\theta^\Gamma} :\G \to S *_{\theta^\Gamma,\sigma}\G$ via $\psi$, that is,
    \begin{align*}
        \Gamma(g) = \psi(\Gamma_{\theta^\Gamma}(g)),
    \end{align*}
    for all $g \in \G$. If $M = \Gamma(\G)$, then $\Gamma$ is recovered together with $M$.
    
    Reciprocally, if $(\theta,\sigma)$ is a twisted partial action of $\G$ on a $\mathbb{K}$-cancellative monoid $T$ with $T_g = T1_g$ and $S = \Gamma_\theta(G)$, there is a $\mathbb{K}$-monoid monomorphism $\phi : S \to T$ such that for all $g \in \G$, $\phi|_{S_g} : S_g \to T_g$ is a $\mathbb{K}$-monoid monomorphism and
    \begin{align*}
        \phi(\theta_g^{\Gamma_\theta}(a)) = \theta_g(\phi(a)),  
    \end{align*}
    for all $a \in S_{g^{-1}}$. 
    
    Furthermore, if $T$ is generated by the elements of the form $\ell1_g$, $\ell \in \mathbb{K}$, then $\phi$ is an isomorphism that restricts to isomorphisms $\phi|_{S_g} : S_g \to T_g$, so $(\theta, \sigma)$ is recovered from $(\theta^{\Gamma_\theta},\sigma)$ via $\phi$.
\end{prop}

\begin{proof}
    Notice first that $\Gamma_{\theta^\Gamma}(g) = n_g\delta_g \in S *_{\theta^\Gamma, \sigma}\G$. Therefore, $\psi(\Gamma_{\theta^\Gamma}(g)) = \psi(n_g\delta_g) = n_g\Gamma(g) = \Gamma(g)$.
    
    For the converse, consider the partial projective representation $\Gamma_\theta :\G \to T *_{\theta,\sigma}\G$. Then
    \begin{align*}
        n_g & = \Gamma_{\theta}(g)\Gamma_{\theta}(g^{-1})\sigma(g^{-1},g)^{-1} = 1_g\delta_g \cdot 1_{g^{-1}}\delta_{g^{-1}} \cdot \sigma(g^{-1},g)^{-1} \\
        & = \theta_g(\theta_{g^{-1}}(1_g)1_{g^{-1}})\sigma(g,g^{-1})\delta_{r(x)} \sigma(g^{-1},g)^{-1} = \theta_g(1_{g^{-1}})\delta_{r(g)} = 1_g\delta_{r(g)}.
    \end{align*}
    
    Since $S$ is generated by the elements $\ell n_g$, $\ell \in \mathbb{K}, g \in \G$, we have that $S_g = \overline{T}_g\delta_{r(g)} \subseteq T *_{\theta,\sigma}\G$, for all $e \in \G_0$, where $\overline{T}$ is the submonoid of $T$ generated by the elements $\ell 1_g$, $\ell \in \mathbb{K}, g \in \G$ and $\overline{T}_g = \overline{T} \cap T_g$. Clearly we have isomorphisms of $\mathbb{K}$-monoids 
    \begin{align*}
        i : S_e & \to \overline{T}_e1_e \\
        a\delta_e & \mapsto a1_e = a,
    \end{align*}
    where $e \in \G_0$, that restricts to isomorphisms of $\mathbb{K}$-monoids
    \begin{align*}
        i : S_g = Sn_g = \overline{T}_g1_g\delta_{r(g)} & \to \overline{T}_g1_g \subseteq T_g
    \end{align*}
    for all $g \in \G$.
    
    Consider $\phi$ the composition between $i$ and the inclusion $\overline{T} \subseteq T$. For $a \in S_{g^{-1}}$, write $a = b\delta_{d(g)}$ with $b \in \overline{T}1_{g^{-1}} \subseteq T_{g^{-1}}$. Then
    \begin{align*}
        \theta_g^{\Gamma_\theta}(a) = \theta_g^{\Gamma_\theta}(b\delta_{d(g)}) & = \Gamma_{\theta}(g)b\delta_{d(g)}\Gamma_\theta(g^{-1})\sigma(g^{-1},g)^{-1} = 1_g\delta_g(b\delta_{d(g)}1_{g^{-1}}\delta_{g^{-1}})\sigma(g^{-1},g)^{-1} \\
        & = 1_g\delta_gb\delta_{g^{-1}}\sigma(g^{-1},g)^{-1} = \theta_g(\theta_{g^{-1}}(1_g)b)\delta_{r(x)} = \theta_g(b)\delta_{r(x)} = \theta_g(\phi(a))\delta_{r(x)}.
    \end{align*}
    
    Hence $\phi(\theta_g^{\Gamma_\theta}(a)) = \phi(\theta_g(\phi(a))\delta_{r(g)}) = \theta_g(\phi(a))$.
\end{proof}

\begin{defi}
Let $\Gamma \colon \G \to M$ and $\Gamma' \colon \G \to M'$ be partial projective representations. A \emph{morphism} from $\Gamma$ to $\Gamma'$ is a $\mathbb{K}$-monoid homomorphism $\varphi \colon M \to M'$ such that $\varphi \circ \Gamma = \Gamma'$. If $\varphi$ is an epimorphism, we say that $\Gamma \to \Gamma'$ is an \emph{epimorphism of partial projective representations}. Monomorphisms and isomorphisms are defined analogously.
\end{defi}

\begin{exe}
The morphism $\psi$ defined in Proposition \ref{proppprtotpa} is an epimorphism of partial projective representations.
\end{exe}

Applying a partial projective representation morphism $\varphi : M \to M'$ in
\begin{align*}
    \Gamma(g^{-1})\Gamma(g)\Gamma(h) = \Gamma(g^{-1})\Gamma(gh) \sigma(g,h),
\end{align*}
we get
\begin{align*}
    \Gamma'(g^{-1})\Gamma'(gh) \sigma'(g,h) = \Gamma'(g^{-1})\Gamma'(h)\Gamma'(h) = \Gamma'(g^{-1})\Gamma'(gh) \sigma(g,h),
\end{align*}
for all $(g,h) \in \G^{(2)}$.

Since $M'$ is $\mathbb{K}$-cancellative, it follows that $\sigma(g,h) = \sigma'(g,h)$, for all $(g,h) \in \text{dom} \sigma'$. In particular, dom$\sigma' \subseteq \text{dom} \sigma$.

\begin{defi}
We say that a partial projective representation $\Gamma$ is \emph{adjusted} if $M = \Gamma(\G)$. We say that $\Gamma$ is \emph{separating} if, denoting by $S$ the subsemigroup of $\Gamma(\G)$ generated by $\ell 1_g$, $\ell \in \mathbb{K}, g \in \G$ and $S_g = Sn_g$, we have that $S_g\Gamma(g) \cap S_h\Gamma(h) = 0$, $g \neq h$. We will denote by
\begin{enumerate}
    \item[(i)] $\mathcal{P}pr\G$ the category which the objects are partial projective representations of $\G$ and the morphisms are morphism of partial projective representations;
    
    \item[(ii)] $\mathcal{AP}pr\G$ the subcategory of $\mathcal{P}pr\G$ which objects are adjusted partial projective representations;
    
    \item[(iii)] $\mathcal{SP}pr\G$ the subcategory of $\mathcal{P}pr\G$ which objects are separating partial projective representations.
\end{enumerate}
\end{defi}

Notice that Hom$_{\mathcal{AP}prG}(\Gamma,\Gamma')$ is either empty or has a unique element that is an epimorphism.

\begin{defi}
Let $\theta = (T_g, \theta_g)_{g \in \G}$, $\theta' = (T_g', \theta_g')_{g \in \G}$ be partial actions of $\G$ in $\mathbb{K}$-cancellative monoids $T$ and $T'$, respectively. We say that a $\mathbb{K}$-monoid homomorphism $\varphi : T \to T'$ is a \emph{morphism of $\theta$ on $\theta'$} if $\varphi(T_g) \subseteq T'_g$, for all $g \in \G$, and the diagram below is commutative.
\end{defi}

\begin{center}
\begin{tikzcd}
T_{g^{-1}} \arrow{rr}{\theta_g} \arrow{d}{\varphi} &  & T_g \arrow{d}{\varphi} \\
T_{g^{-1}}' \arrow{rr}{\theta_g'}                    &  & T_g' 
\end{tikzcd}    
\end{center}

If $\sigma$ and $\sigma'$ are $\mathbb{K}$-valued twistings associated with $\theta$ and $\theta'$ respectively, we say that a morphism $\varphi : \theta \to \theta'$ is a twisted partial action morphism if
\begin{align*}
    \varphi(T_g \cap T_{gh}) \neq 0 \implies \sigma(g,h) = \sigma'(g,h)
\end{align*}
for all $(g,h) \in \G^{(2)}$.

Since $T_g = T1_g$ and $T_g' = T'1_g'$, we have that $T_g = 0 \implies T_g' = 0$ and more generally
\begin{align*}
    T_{g_1} \cap T_{g_2} \cap \cdots \cap T_{g_m} = 0 \implies T_{g_1}' \cap T_{g_2}' \cap \cdots \cap T_{g_m}' = 0.
\end{align*}

In particular, $T_g' \cap T_{gh}' \neq 0 \implies T_g \cap T_{gh} \neq 0$ when $(g,h) \in \G^{(2)}$, that is, dom$ \sigma' \subseteq \text{dom} \sigma$. 

Notice that if $\varphi$ is an monomorphism, then
\begin{align*}
    T_{g_1} \cap T_{g_2} \cap \cdots \cap T_{g_m} = 0 \iff T_{g_1}' \cap T_{g_2}' \cap \cdots \cap T_{g_m}' = 0.
\end{align*}

A twisted partial action of $\G$ on $T$ is \emph{adjusted} if $T = \bar T$ as in Proposition \ref{proptpatoppr}. Every twisted partial action $\theta$ can become adjusted replacing $T$ by $\bar T$, and then we will denote this new twisted partial action by $\bar \theta$. In particular, if $A$ is a $\mathbb{K}$-cancellative $\mathbb{K}$-algebra such that the $\mathbb{K}$-monoid $(A,\cdot)$ is adjusted, the last assumption in Proposition \ref{propcompalgmon} holds.

\begin{defi}
We will denote by
\begin{enumerate}
    \item[(i)] $\mathcal{TP}a\G$ the category of twisted partial actions of $\G$ on $\mathbb{K}$-cancellative monoids and twisted partial action morphisms;
    
    \item[(ii)] $\mathcal{ATP}a\G$ the subcategory of $\mathcal{TP}a\G$ of the adjusted twisted partial actions of $\G$.
\end{enumerate}
\end{defi}

As in the case of adjusted partial projective representations, the set $\text{Hom}_{\mathcal{ATP}a\G}(\theta,\theta')$ is either empty or consists of a single element, which is necessarily an epimorphism. We can now state the main result of this section. The theorem below is the groupoid analogue of \cite[Theorem 3]{dokuchaev2012partialprojectiverep2}. In view of Propositions \ref{proppprtotpa} and \ref{propcomp}, most of its proof carries over with only minor modifications.

\begin{theorem}\label{teoprincipal}
    \begin{enumerate}
        \item[(i)] There is a functor $\mathcal{P}pr\G \to \mathcal{ATP}a\G$ that takes any $\Gamma \in \text{\emph{Ob}} \mathcal{P}pr\G$ to $\theta^\Gamma$.
        
        \item[(ii)] There is a functor $\mathcal{TP}a\G \to \mathcal{SP}pr\G$ that takes $\theta \in \text{\emph{Ob}} \mathcal{TP}a\G$ to $\Gamma_\theta$. Furthermore, if $\theta$ is adjusted, then $\Gamma_\theta$ is adjusted.
        
        \item[(iii)] For all $\Gamma \in \text{\emph{Ob}} \mathcal{P}pr\G$ there is a morphism $\Gamma_{\theta^\Gamma} \to \Gamma$ that is an epimorphism if $\Gamma$ is adjusted and is a monomorphism if $\Gamma$ is separating.
        
        \item[(iv)] For all $\theta \in \text{\emph{Ob}} \mathcal{TP}a\G$ there is a monomorphism $\theta^{\Gamma_\theta} \to \theta$ that is an isomorphism if $\theta$ is adjusted. In any case, $\theta^{\Gamma_\theta} \cong \bar \theta$.
        
        \item[(v)] The restriction of the functors in (i) and (ii) form an equivalence between the categories $\mathcal{AP}pr\G \cap \mathcal{SP}pr\G$ and $\mathcal{ATP}a\G$.
        
        \item[(vi)] The restriction of the functor from (i) to $\mathcal{AP}pr\G \to \mathcal{ATP}a\G$ is right adjoint to the restriction of the functor from (ii) to $\mathcal{ATP}a\G \to \mathcal{AP}pr\G$.
    \end{enumerate}
\end{theorem}
\begin{proof}
    (i): Let $(\Gamma, M) \in \text{Ob} \mathcal{P}pr\G$ with factor set $\sigma$. We have that $(\theta^\Gamma,\sigma)$ is a twisted partial action of $\G$ on $S$ by Proposition \ref{proppprtotpa}. Since $S$ is generated by the elements $\ell n_g$, $\ell \in \mathbb{K}, g \in \G$, $\theta^\Gamma$ is adjusted.
    
    Let $(\Gamma',M') \in \text{Ob} \mathcal{P}pr\G$ with factor set $\sigma'$ and $\varphi \in \text{Hom}_{\mathcal{P}pr\G}(\Gamma, \Gamma')$. We have that $(\theta^{\Gamma'},\sigma')$ is a twisted partial action of $\G$ on $S' = \varphi(S)$. Moreover, $$\varphi(n_g) = \varphi(\Gamma(g)\Gamma(g^{-1})\sigma(g^{-1},g)^{-1}) = \Gamma'(g)\Gamma'(g^{-1})\sigma'(g^{-1},g)^{-1} = n_g',$$ from where it follows that $\varphi(S_g) = S_g'$, for all $g \in \G$. We have that $\varphi(\theta^\Gamma(a)) = \theta^{\Gamma'}(\varphi(a))$, for all $a \in S_{g^{-1}}$, which implies that $\varphi : S \to S'$ is a partial action morphism such that $\varphi|_{S_g} : S_g \to S_g'$ is a monoid morphism for all $g \in \G$.
    
    We already observed that $\sigma(g,h) = \sigma'(g,h)$, for all $(g,h) \in \text{dom} \sigma'$. Since
    \begin{align*}
        \varphi(S_g \cap S_{gh}) = \varphi(Sn_gn_{gh}) = S'n'_gn'_{gh} = S'_g \cap S'_{gh},
    \end{align*}
    for $(g,h) \in \G^{(2)}$, we have that $\varphi(S_g \cap S_{gh}) \neq 0 \implies S'_g \cap S'_{gh} \neq 0 \implies (g,h) \in \text{dom} \sigma'$.

    So the maps
    \begin{align*}
        \text{Ob}\mathcal{P}pr\G & \to \text{Ob}\mathcal{ATP}a\G \\
        \Gamma & \mapsto \theta^\Gamma
    \end{align*} 
    and
    \begin{align*}
        \text{Hom}\mathcal{P}pr\G & \to \text{Hom}Ob\mathcal{ATP}a\G \\
        \varphi & \mapsto \varphi
    \end{align*}
    described above define a functor from $\mathcal{P}pr\G$ to $\mathcal{ATP}a\G$. 
    
    (ii): Let $(\theta, \sigma) \in \text{Ob} \mathcal{TP}a\G$ be a twisted partial action of $\G$ on $T$. We already know that $\Gamma_\theta :\G \to T *_{\theta,\sigma}\G$ is a separating partial projective representation with factor set $\sigma$. Let $(\theta', \sigma') \in \text{Ob} \mathcal{TP}a\G$ be a twisted partial action of $\G$ on $T'$ and $\varphi : (\theta, \sigma) \to (\theta',\sigma')$ a twisted partial action morphism. Since $\varphi(T_g) \subseteq T'_g$, we can consider the map
    \begin{align*}
        \Phi : T *_{\theta,\sigma}\G & \to T' *_{\theta',\sigma'}\G \\
        a\delta_g & \mapsto \varphi(a)\delta_g'.
    \end{align*}
    
    Let $a\delta_g, b\delta_h \in T *_{\theta,\sigma}\G$. We have that
    \begin{align*}
        \Phi(a\delta_g)\Phi(b\delta_h)  = \varphi(a)\delta_g'\varphi(b)\delta_h' & = \begin{cases} \theta'_g(\theta'_{g^{-1}}(\varphi(a))\varphi(b))\sigma'(g,h)\delta_{gh}', \text{ if } (g,h) \in \G^{(2)} \\
        0, \text{ otherwise}
        \end{cases} \\
        & = \begin{cases}
        \theta_g'(\varphi(\theta_{g^{-1}}(a)b))\sigma'(g,h) \delta'_{gh} , \text{ if } (g,h) \in \G^{(2)} \\
        0, \text{ otherwise}
        \end{cases} \\
        & = \begin{cases} 
        \varphi(\theta_g(\theta_{g^{-1}}(a)b))\sigma'(g,h)\delta_{gh}', \text{ if } (g,h) \in \G^{(2)} \\
        0, \text{ otherwise.}
        \end{cases}
    \end{align*}
    
    On the other hand,
    \begin{align*}
        \Phi(a\delta_gb\delta_h) & = \begin{cases} \Phi(\theta_g(\theta_{g^{-1}}(a)b)\sigma(g,h)\delta_{gh}), \text{ if } (g,h) \in \G^{(2)} \\
        0, \text{ otherwise}
        \end{cases} \\
        & = \begin{cases} \varphi(\theta_g(\theta_{g^{-1}}(a)b))\sigma(g,h)\delta'_{gh}, \text{ if } (g,h) \in \G^{(2)} \\
        0, \text{ otherwise.}
        \end{cases}
    \end{align*}
    
    If $(g,h) \in \G^{(2)}$ and $\varphi(T_g \cap T_{gh}) \neq 0$, then $\sigma(g,h) = \sigma'(g,h)$. If $(g,h) \in \G^{(2)}$ and $\varphi(T_g \cap T_{gh}) = 0$, then $\varphi(\theta_g(\theta_{g^{-1}}(a)b)) = 0$. If $(g,h) \notin \G^{(2)}$, then $\sigma(g,h) = \sigma'(g,h) = 0$. In this way we guarantee that $\Phi$ is a $\mathbb{K}$-monoid homomorphism. Besides that,
    \begin{align*}
        \Phi(\Gamma_\theta(g)) = \Phi(1_g\delta_g) = \varphi(1_g)\delta_g' = 1_g'\delta_g' = \Gamma_{\theta'}(g),
    \end{align*}
    so $\Phi$ is a morphism between $\Gamma_\theta$ and $\Gamma_{\theta'}$.

    Thus the maps
    \begin{align*}
        \text{Ob}\mathcal{TP}a\G & \to \text{Ob}\mathcal{SP}pr\G \\
        \theta & \mapsto \Gamma_\theta
    \end{align*}
    and
    \begin{align*}
        \text{Hom}\mathcal{TP}a\G & \to \text{Hom}\mathcal{SP}pr\G \\
        \varphi & \mapsto \Phi
    \end{align*}
    described above define the functor of the first statement.
    
    For the second statement, assume that $\theta$ is adjusted. Consider arbitrary $0 \neq a\delta_g \in T *_{\theta,\sigma}\G$. We have that $a = \ell 1_{g}1_{h_1} \cdots 1_{h_m}$ with $r(h_i) = r(g)$, for all $1 \leq i \leq m$. Hence 
    \begin{align*}
        a\delta_g = \ell 1_{h_1} \cdots 1_{h_m} \delta_{r(g)} 1_g\delta_g = \ell 1_{h_1} \delta_{r(g)} \cdots 1_{h_m} \delta_{r(g)} 1_g\delta_g.
    \end{align*}
    
    Notice that for all $k \in r(g)\G$ we have
    \begin{align*}
        1_k\delta_{r(g)} = \sigma(k^{-1},k)^{-1} 1_{k}\delta_{k} 1_{k^{-1}}\delta_{k^{-1}} = \sigma(k^{-1},k)^{-1} \Gamma_\theta(k)\Gamma_\theta(k^{-1}),
    \end{align*}
    thus
    \begin{align*}
        a\delta_g = \beta \Gamma_\theta(h_1)\Gamma_\theta(h_1^{-1}) \cdots \Gamma_\theta(h_m)\Gamma_\theta(h_m^{-1})\Gamma_\theta(g),
    \end{align*}
    that is, $\Gamma_\theta$ is adjusted. 
    
    (iii): Let $\Gamma \in \text{Ob} \mathcal{P}pr\G$. The morphism $\psi : \Gamma_{\theta^\Gamma} \to \Gamma$ defined in Proposition \ref{proppprtotpa} is an epimorphism if $\Gamma$ is adjusted by definition. Assume that $\Gamma$ is separating. Clearly $\psi(a\delta_g) \neq \psi(b\delta_h)$, for all $g \neq h$, $a \in S_g$ and $b \in S_h$. If $\psi(a\delta_g) = \psi(b\delta_g)$ with $a,b \in S_g \neq 0$, then $a\Gamma(g) = b\Gamma(g)$ and, multiplying both sides by $\Gamma(g^{-1})$ we obtain $an_g\sigma(g,g^{-1}) = bn_g\sigma(g,g^{-1})$. Now, $a,b \in S_g$ yields $an_g = a$ and $bn_g = b$. Since $S$ is $\mathbb{K}$-cancellative, it follows that $a = b$.
    
    (iv): Let $(\theta, \sigma) \in \text{Ob} \mathcal{TP}a\G$ be a twisted partial action of $\G$ on $T$. It follows from Proposition \ref{propcomp} that the morphism $\phi$ is the required isomorphism between $\theta^{\Gamma_\theta}$ and $\bar \theta$.
    
    (v): Just notice that the set of morphisms between any two objects of $\mathcal{AP}pr\G \cap \mathcal{SP}pr\G$ is empty or has only one element and the same holds in $\mathcal{ATP}a\G$, so the restrictions are full and faithful. The result now follows from (iii) and (iv).
    
    (vi): We already know that the domains and codomains of these functors are appropriate, as well as the sets of morphisms between any two objects of one of these categories have 0 or 1 element only. Moreover, by the same arguments given in \cite[Theorem 3(vi)]{dokuchaev2012partialprojectiverep2}, we have that for $\Gamma \in \text{Ob} \mathcal{AP}pr\G$ and $\theta \in \text{Ob} \mathcal{ATP}a\G$, 
    \begin{align*}
        \# \text{Hom}_{\mathcal{AP}pr\G} (\Gamma_\theta,\Gamma) = \# \text{Hom}_{\mathcal{ATP}a\G} (\theta,\theta^\Gamma),
    \end{align*} which concludes the proof.
\end{proof}

 \section*{Declarations}
    
    \subsection*{Author's Contribution}
    All authors wrote the main manuscript text and reviewed the manuscript.
    
    \subsection*{Competing Interests}
    The authors have no competing interests as defined by Springer, or other interests that might be perceived to influence the results and/or discussion reported in this paper.
    
    \subsection*{Ethical Approval}
    Not applicable.
    
    \subsection*{Availability of Data and Materials}
    Data sharing not applicable to this article as no datasets were generated or analysed during the current study.
    
    \subsection*{Funding}
     T. Tamusiunas was partially supported by CNPq (Brazil) through a Productivity Research Fellowship, Grant No. 303411/2025-2, and by CNPq (Brazil) under the Universal Call, Grant No. 403606/2025-0.
\bibliographystyle{abbrvnat}
    \footnotesize{\bibliography{bibliografia}}

\end{document}